\documentclass[12pt,a4paper,reqno]{amsart}
\usepackage[mathcal]{eucal}
\usepackage{amssymb}
\usepackage[nice]{nicefrac}
\usepackage{hyperref}
\usepackage[normalem]{ulem}
\usepackage{xcolor}
\usepackage{mathrsfs}  
\usepackage{cite}

\theoremstyle{plain}
\newtheorem{theorem}{Theorem}[section]
\newtheorem{lemma}[theorem]{Lemma}
\newtheorem{proposition}[theorem]{Proposition}
\newtheorem{corollary}[theorem]{Corollary}

\theoremstyle{definition}

\newtheorem{problem}[theorem]{Problem}

\numberwithin{equation}{section}

\DeclareMathOperator{\hdim}{hdim}
\DeclareMathOperator{\hspec}{hspec}

\makeatletter
\newcommand{\AlignFootnote}[1]{%
    \ifmeasuring@
    \else
        \footnote{#1}%
    \fi
}
\makeatother

\newcommand{\Z}{\mathbb{Z}}
\newcommand{\Q}{\mathbb{Q}}
\newcommand{\N}{\mathbb{N}}

\title[The finitely generated spectra of a family of pro-$p$ groups]{The finitely generated Hausdorff spectra of a family of pro-$p$ groups}

\author[I. de las Heras]{Iker de las Heras} 
\address{Iker de las Heras: Mathematisches Institut, Heinrich-Heine-Universit\"at, 40225
  D\"usseldorf, Germany; Department of Mathematics, University
  of the Basque Country UPV/EHU, 48940 Leioa, Spain}
\email{Iker.Delasheras@hhu.de}

\author[A. Thillaisundaram]{Anitha Thillaisundaram} 
\address{Anitha Thillaisundaram: School of Mathematics and Physics,
  University of Lincoln, Lincoln LN6 7TS, United Kingdom}
\email{anitha.t@cantab.net}

\date{\today}

\thanks{The first author is supported by the Spanish Government, grant MTM2017-86802-P, partly with FEDER funds, and by the Basque Government, grant IT974-16. He is also supported by a postdoctoral grant of the Basque Government, and from the Deutsche Forschungsgemeinschaft grant GRK 2240: Algebro-Geometric Methods in Algebra, Arithmetic and Topology. The second author acknowledges the support from EPSRC, grant EP/T005068/1. Both authors thank the Heinrich-Heine-Universit\"{a}t D\"{u}sseldorf, where a large part of this research was carried out.}



\begin{document}

\maketitle

\begin{abstract}
 Recently the first example of a family of pro-$p$ groups, for $p$ a prime, with full normal Hausdorff spectrum was constructed. In this paper we further investigate this family by computing their finitely generated Hausdorff spectrum with respect to each of the five standard filtration series: the $p$-power series, the iterated $p$-power series, the lower $p$-series, the Frattini series and the dimension subgroup series. Here the finitely generated Hausdorff spectra of these groups consist of infinitely many rational numbers, and their computation requires a rather technical approach. This result also gives further evidence to the non-existence of a finitely generated pro-$p$ group with uncountable finitely generated Hausdorff spectrum.
\end{abstract}

\section{Introduction}

Let $\Gamma$ be a countably based infinite
profinite group and consider a \emph{filtration series} $\mathcal{S}$
of~$\Gamma$, i.e. a descending chain
$\Gamma = \Gamma_0 \ge \Gamma_1 \ge \cdots$ of open normal subgroups ${\Gamma_i \trianglelefteq_\mathrm{o} \Gamma}$ with $\bigcap_i \Gamma_i = 1$.  These open normal subgroups yield a base of neighbourhoods of the identity and
induce a translation-invariant metric on~$\Gamma$ which is given by
$d^\mathcal{S}(x,y) = \inf \left\{ \lvert \Gamma : \Gamma_i \rvert^{-1} \mid x
  \equiv y \pmod{\Gamma_i} \right\}$,
for $x,y \in \Gamma$.  This  gives,  for a subset
$U \subseteq \Gamma$, the \emph{Hausdorff
  dimension} $\hdim_{\Gamma}^\mathcal{S}(U) \in [0,1]$ with respect to the filtration series~$\mathcal{S}$.

Recently there has been much interest concerning Hausdorff dimensions in profinite groups, starting with the pioneering work of Abercrombie~\cite{Abercrombie} and of Barnea and Shalev~\cite{BaSh97}; see for example \cite{AbVi05,BaKl03,BaVa19,Er04,Er10,FeGiGo,FAZR,FeZu14,GGK,Gl14,GoZo,JaKl07,KTZR,Sunic}.
%we refer the reader to~\cite{KTZR} for an overview. 
Barnea and Shalev~\cite{BaSh97} proved the following
group-theoretic formula of the Hausdorff dimension with respect to~$\mathcal{S}$ of a
closed subgroup~$H$ of~$\Gamma$ as a logarithmic density:
\begin{equation*}
  \hdim_{\Gamma}^\mathcal{S}(H) =
  \varliminf_{i\to \infty} \frac{\log \lvert H\Gamma_i : \Gamma_i
    \rvert}{\log \lvert \Gamma : \Gamma_i \rvert},
\end{equation*}
where 
$\varliminf_{i \to \infty} a_i$ is the lower limit of a
sequence $(a_i)_{i \in \mathbb{N}}$ in
$\mathbb{R}$.

As observed in~\cite{KTZR}, the Hausdorff dimension function depends on %is sensitive to 
the choice of the filtration series~$\mathcal{S}$, hence for a pro-$p$ group~$\Gamma$, where $p$ is a prime, it makes sense to restrict our attention to the five standard filtration series below:
\begin{enumerate}
\item[$\bullet$] the \emph{$p$-power
  series}~$\mathcal{P}$ of $\Gamma$, which is given by
\[
\mathcal{P} \colon \Gamma^{p^i} = \langle x^{p^i} \mid x \in \Gamma
\rangle, \quad i \ge 0;
\]
\item[$\bullet$] the \emph{iterated $p$-power series}~$\mathcal{I}$ of~$\Gamma$, which is defined by 
\[
\mathcal{I} : I_0(\Gamma)=\Gamma,\qquad I_i(\Gamma)=I_{i-1}(\Gamma)^p,\quad \text{for }i\ge 1;
\]
\item[$\bullet$]
the \emph{lower $p$-series}~$\mathcal{L}$ (or lower $p$-central
series) of $\Gamma$, which is given recursively by
\[
  \mathcal{L} \colon P_0(\Gamma) = \Gamma,  %
   \quad  P_i(\Gamma) = P_{i-1}(\Gamma)^p
    \, [P_{i-1}(\Gamma),\Gamma] \quad
    \text{for $i \ge 1$;}
    \]
\item[$\bullet$] the \emph{Frattini series}~$\mathcal{F}$ of $\Gamma$, which is
  given recursively by
  \[
 \qquad \mathcal{F} \colon \Phi_0(\Gamma) = \Gamma, %
   \qquad\Phi_i(\Gamma) = \Phi_{i-1}(\Gamma)^p
    \, [\Phi_{i-1}(\Gamma),\Phi_{i-1}(\Gamma)] \quad \text{for $i \ge 1$;} 
    \]
  \item[$\bullet$] the (modular) \emph{dimension subgroup series}~$\mathcal{D}$ (or
  Jennings series or Zassenhaus series) of $\Gamma$, which is defined recursively
  by
  \[
  \qquad\mathcal{D} \colon D_0(\Gamma) = \Gamma, %
   \qquad D_i(\Gamma) = D_{\lceil i/p \rceil}(\Gamma)^p 
    \prod_{0 \le j <i} [D_j(\Gamma),D_{i-j}(\Gamma)] \quad \text{for $i \ge 1$.}  
\]
\end{enumerate}
Recall
for a pro-2 group~$\Gamma$, the iterated $2$-power series coincides with the Frattini series.

It is often of interest to study the collection of Hausdorff dimensions in a given profinite group~$\Gamma$, which gives rise to the following definition:
the \emph{Hausdorff spectrum} of $\Gamma$, with respect to $\mathcal{S}$,
is
\[
\hspec^\mathcal{S}(\Gamma) = \{ \hdim_{\Gamma}^\mathcal{S}(H) \mid H \le_\mathrm{c} \Gamma\}
\subseteq [0,1],
\]
where $H$ runs through all closed subgroups of~$\Gamma$.  There are two well-known restricted Hausdorff spectra: the normal Hausdorff spectrum and the finitely generated Hausdorff spectrum. Here the
\emph{normal Hausdorff spectrum} of~$\Gamma$ with respect to
$\mathcal{S}$ is
\[
\hspec^{\mathcal{S}}_{\trianglelefteq}(\Gamma) = \{ \hdim^{\mathcal{S}}_{\Gamma}(H) \mid H
\trianglelefteq_\mathrm{c} \Gamma \};
\]
and the \emph{finitely generated
  Hausdorff spectrum}  with respect to
$\mathcal{S}$ is defined as
\[
\hspec^{\mathcal{S}}_{\text{fg}}(\Gamma) = \{ \hdim^{\mathcal{S}}_\Gamma(H) \mid
H \le_\mathrm{c} \Gamma \text{ and } H \text{ finitely generated} \}.
\]

Over the past few years, the study of the normal Hausdorff spectra of finitely generated pro-$p$
groups has received quite a bit of attention. Indeed, the first examples of finitely generated pro-$p$ groups with infinite normal Hausdorff spectra, with respect to the series $\mathcal{P}$, $\mathcal{I}$, $\mathcal{L}$, $\mathcal{F}$, $\mathcal{D}$ were constructed in~\cite{KT}, and the first family of  finitely generated pro-$p$ groups~$\mathfrak{G}(p)$  with full normal Hausdorff spectra $[0,1]$ was constructed in~\cite{HK} and \cite{HT}.

Those pro-$p$ groups~$\mathfrak{G}(p)$ are $2$-generated extensions of an elementary abelian pro-$p$ group by the pro-$p$ wreath product $W= C_p \mathrel{\hat{\wr}} \mathbb{Z}_p=\varprojlim_{k \in \N} C_p \mathrel{\wr} C_{p^k}$. In this paper, we further investigate the pro-$p$ groups~$\mathfrak{G}(p)$ by computing their finitely generated Hausdorff spectra.
As with the  normal Hausdorff spectra, computations of the finitely generated Hausdorff spectra have only recently appeared in the literature; see~\cite{F, KT, GGK}. We prove the following:

\begin{theorem}
\label{thm:main}
For $p$ a prime, the pro-$p$ group $\mathfrak{G}(p)$ satisfies 
\[
\hspec_{\textup{fg}}^{\mathcal{S}}(\mathfrak{G}(p))=   
\{\nicefrac{d^2}{p^{2l}} \mid  l\in \mathbb{N},\, 0\le d\le p^l\big\}  
 \]
 for $\mathcal{S}\in\{\mathcal{L},\mathcal{D},\mathcal{P}, \mathcal{I}, \mathcal{F},\mathcal{M}
\}$. 
\end{theorem}

\noindent Here $\mathcal{M}$ denotes a natural filtration series that arises from the construction of~$\mathfrak{G}(p)$; see Section~2 for details.

We note that for $p=2$ the construction of the pro-$2$ group~$\mathfrak{G}(2)$ (which was given in~\cite{HT}) in the above family is different from the odd prime case (given in~\cite{HK}); see Section~2 for precise details.

Now to prove Theorem~\ref{thm:main}, we will focus solely on the pro-$2$ group~$\mathfrak{G}(2)$. This is because a more delicate treatment is required for~$\mathfrak{G}(2)$, as opposed to the odd $p$ case. We will see that the proof of Theorem~\ref{thm:main} for the pro-$2$ group~$\mathfrak{G}(2)$ simplifies in a straightforward manner for the pro-$p$ groups~$\mathfrak{G}(p)$ for odd~$p$, except for the iterated $p$-power series~$\mathcal{I}$, which requires a separate treatment. This will be done in Section~\ref{sec:iterated}.

As seen below, the computation of the finitely generated Hausdorff spectra in general requires a greater level of technical machinery, as opposed to the computation of the normal Hausdorff spectra. In particular, precise information about the terms of the filtration series is required. As illustrated in the proof of Theorem~\ref{thm:main}, we introduce the idea of counting in what we call blocks.

In view of the main result of this paper and of other results concerning the finitely generated Hausdorff spectra, it is natural to ask the following:
\begin{problem}
Does there exist a finitely generated \mbox{pro-$p$} group  with uncountable finitely generated Hausdorff spectra with respect to one or several of the five standard filtration series?
\end{problem}

\medskip

\noindent \emph{Organisation}.  %Section~$2$ contains preliminary results.  
In Section~\ref{sec:3} we recall the construction   of the pro-$p$ groups~$\mathfrak{G}(p)$, for $p$ a prime, and we recall several properties of~$\mathfrak{G}(2)$. In Section~\ref{sec:normal-spectra}, we  determine, as  precisely as possible, the $2$-power series and the Frattini series of~$\mathfrak{G}(2)$, before computing in Section~\ref{sec:fg-spec} the finitely generated Hausdorff spectra of~$\mathfrak{G}(2)$ with respect to~$\mathcal{M}$, $\mathcal{L}$, $\mathcal{D}$,  $\mathcal{P}$ and $\mathcal{F}$. Finally in Section~\ref{sec:iterated},   we compute the finitely generated Hausdorff
spectrum of~$\mathfrak{G}(p)$, for odd primes~$p$, with respect to the iterated $p$-power series~$\mathcal{I}$, which completes the proof of Theorem~\ref{thm:main}.

\medskip

\noindent \textit{Notation.}     All subgroups of profinite
groups are generally taken to be closed subgroups.  We use the notation $\le_\text{o}$ and $\le_\text{c}$ to denote open and closed subgroups respectively. 
Throughout, we  use
left-normed commutators, for example, $[x,y,z] = [[x,y],z]$.

%%%%

%\section{Preliminaries}

%{\color{magenta}
%It would be nice if we could somehow state the method that we use in the proof regarding blocks in such a way that it works for all groups of certain type. In that way not only we could add it to Section Preliminaries, but also we could motivate the paper referring to this new method to compute the finitely generated Hausdorff spectra.

%\bigskip
%}

%%%%%

\section{The family of pro-\texorpdfstring{$p$}{p} groups \texorpdfstring{$\mathfrak{G}(p)$}{G(p)}}\label{sec:3}

Let $p$ be any prime. For $k\in\N$, let $\langle \dot x_k\rangle\cong C_{p^k}$ and $\langle \dot y_k\rangle\cong C_p$.
Define
$$W_k=\langle \dot y_k\rangle\wr\langle \dot x_k\rangle\cong B_k\rtimes\langle \dot x_k\rangle$$
where $B_k = \prod_{i = 0}^{p^k-1} \langle {\dot y_k}^{\,\dot  x_k^{\, i}}\rangle\cong C_p^{\, p^k}$.
The structural results for the finite wreath products $W_k$ transfer naturally to the inverse limit $W\cong\varprojlim_k W_k$ 
which is the pro-$p$ wreath product
$$W=\langle \dot x,\dot y\rangle=B\rtimes \langle\dot x\rangle\cong C_p\ \hat{\wr}\ \Z_p$$
with top group $\langle\dot  x\rangle\cong\Z_p$ and base group $B=\prod_{i\in\Z}\langle \dot y^{\dot x^i}\rangle\cong C_p^{\,\aleph_0}$. We refer the reader to~\cite[\S2.4]{KT} for further results concerning the groups $W_k$ and $W$.

Let $F_2=\langle a,b\rangle$ be the free pro-$p$ group on two generators and let $k\in\mathbb{N}$.
There exists a closed normal subgroup $R\trianglelefteq F_2$, respectively $R_k\trianglelefteq F_2$, such that 
\[
F_2/R\cong W, \quad\text{ respectively }\,\,
F_2/R_k\cong W_k=C_p \,\wr \,C_{p^k},
\]
with $a$ corresponding to~$\dot x$, respectively~$\dot x_k$, and $b$ corresponding to~$\dot y$, respectively~$\dot y_k$.
 
Let $Y\ge R$ be the closed normal subgroup of $F_2$ such that $Y/R$ is
the pre-image of $B$  in $F_2/R$, and let  $Y_k\ge R_k$ be the closed normal subgroup of $F_2$ such that  $Y_k/R_k$ is
the pre-image of $B_k$ in $F_2/R_k$.

As mentioned in the introduction, the definition of the pro-$p$ groups~$\mathfrak{G}(p)$, for $p$ an odd prime, is slightly different from the case $p=2$.
For an odd prime~$p$, the pro-$p$ group~$\mathfrak{G}:=\mathfrak{G}(p)$ is defined as follows:
\[
\mathfrak{G}=F_2/N\qquad
\text{where}\quad 
N=[R,Y]Y^p.
\]
For $k\in\N$ we set 
\[
\mathfrak{G}_k=F_2/N_k\qquad
\text{where}\quad 
 N_k=[R_k,Y_k]Y_k^{\,p}\langle a^{p^k}\rangle^{F_2}.
 \]

We denote by $H$ and $Z$  the closed normal subgroups of~$\mathfrak{G}$ corresponding 
to $Y/N$ and $R/N$, and we denote by $H_k$ and $Z_k$ the closed normal subgroups of~$\mathfrak{G}_k$ corresponding 
to  $Y_k/N_k$ and $R_k/N_k$. We denote the images of $a, b$ in~$\mathfrak{G}$, respectively in~$\mathfrak{G}_k$, by $x,y$, respectively $x_k,y_k$, so that  $\mathfrak{G}=\langle x,y\rangle$ and $\mathfrak{G}_k=\langle x_k,y_k\rangle$.

The groups $\mathfrak{G}_k$ are finite for all $k\in\mathbb{N}$ and  they  form an inverse system giving $\varprojlim_k \mathfrak{G}_k=\mathfrak{G}$.
Furthermore  we have $[H,Z]=H^p=1$, respectively $[H_k,Z_k]=H_k^{\,p}=1$.

For the case $p=2$, we set
\[
N=[R,Y]R^2,
\]
 respectively 
 \[
 N_k=[R_k,Y_k]R_k^{\,2}\langle a^{2^k}\rangle^{F_2},
 \]
  and, for convenience we write $G:=\mathfrak{G}(2)$, and define 
  \[
  G=F_2/N, \quad\text{respectively }\,\,G_k=F_2/N_k.
  \]

As in the odd prime case, we denote by $H$ and $Z$  the closed normal subgroups of $G$ corresponding 
to $Y/N$ and $R/N$, and we denote by $H_k$ and $Z_k$ the closed normal subgroups of  $G_k$ corresponding 
to  $Y_k/N_k$ and $R_k/N_k$. Likewise, we denote the images of $a, b$ in~$G$, respectively in~$G_k$, by $x,y$, respectively $x_k,y_k$, so that  $G=\langle x,y\rangle$ and $G_k=\langle x_k,y_k\rangle$.

Similarly, we have that the groups $G_k$ are finite for all $k\in\mathbb{N}$ and that 
$\varprojlim_k G_k=G$.
Here we have $[H,Z]=Z^2=1$, respectively $[H_k,Z_k]=Z_k^{\,2}=1$.

 \smallskip
 
We recall the following definition from~\cite{KT}: for a countably
based infinite \mbox{pro-$p$} group~$\Gamma$, equipped with a filtration series
$\mathcal{S} \colon \Gamma = \Gamma_0 \ge \Gamma_1\ge \cdots$, and a
closed subgroup $H \le_\mathrm{c} \Gamma$, we say that $H$ has \emph{strong
  Hausdorff dimension in $\Gamma$ with respect to $\mathcal{S}$} if
\[
\hdim^{\mathcal{S}}_{\Gamma}(H) = \lim_{i \to \infty} \frac{\log_p \lvert H
  \Gamma_i : \Gamma_i \rvert}{\log_p \lvert \Gamma : \Gamma_i \rvert}
\]
is given by a proper limit. 

From~\cite{HK} and \cite{HT}, we have
 that for these pro-$p$ groups~$\mathfrak{G}(p)$, for $p$ a prime, the subgroup~$Z$ has strong Hausdorff dimension~$1$ with respect to the  filtration series $\mathcal{M}$, $\mathcal{L}$, $\mathcal{D}$, $\mathcal{P}$, $\mathcal{I}$, and $\mathcal{F}$. Here we recall that $\mathcal{M}:M_0\ge M_1\ge\cdots$ stands for the natural filtration series of~$\mathfrak{G}(p)$ where each $M_i$ is the subgroup of~$\mathfrak{G}(p)$ corresponding to~$N_i/N$, and here $N_0=F_2$.

\smallskip

Recall that for $p=2$, we write $G=\mathfrak{G}(2)$. As observed above, owing to the fact that the subgroup~$H$ of~$G$ is not of exponent~$2$, a more delicate treatment is required for~$G$, as opposed to the odd $p$ case. Also, as mentioned in the introduction, it turns out that the proof of the finitely generated Hausdorff spectra for the pro-$2$ group~$G$, with respect to $\mathcal{M}$, $\mathcal{L}$, $\mathcal{D}$, $\mathcal{P}$, $\mathcal{F}$, generalises immediately to the pro-$p$ groups~$\mathfrak{G}(p)$, for odd primes~$p$. Therefore for the rest of this section and the next two, we focus solely on the pro-$2$ group~$G$, and only in Section~\ref{sec:iterated} do we explicitly consider the pro-$p$ groups~$\mathfrak{G}(p)$, for odd $p$, for computing the finitely generated Hausdorff spectrum of~$\mathfrak{G}(p)$ with respect to the iterated $p$-power series~$\mathcal{I}$.

\subsection{Properties of \texorpdfstring{$G$}{G}}
\label{subsection properties of G_k}

Here we recall some properties of the pro-$2$ group~$G$ from \cite[Sec.~3]{HT}.

For $k\in\N$, the logarithmic order of $G_k$ is
$$
\log_2|G_k|=k+2^{k+1}+\binom{2^k}{2}= k+2^{2k-1}+2^{k+1}-2^{k-1},
$$
and the nilpotency class of $G_k$ is $2^{k+1}-1$.
Also we have that $H_k^{\,2}=Z_k$ and, consequently, that $H^2=Z$. In particular, the exponent of $H_k$, and of $H$, is $4$. 

Further
\[
\log_2|Z:\gamma_i(G)\cap Z|=\begin{cases}
\dfrac{i^2-1}{4}  & \text{ if $i$ is odd},\\
& \vspace{-0.8em}\\
\dfrac{i^2}{4} &  \text{ if $i$ is even}.
\end{cases}
\]

In the sequel, we need the following notation, which will be used frequently in the paper.
We will denote $c_1=y$ and $c_i=[y,x,\overset{i-1}{\ldots},x]$ for $i\in\mathbb{N}_{\ge 2}$.
For $k\in\N$, we will use the same notation for the corresponding elements in the group $G_k$. Thus, we will also write $c_1=y_k$ and $c_i=[y_k,x_k,\overset{i-1}{\ldots},x_k]$ for every $i\in\mathbb{N}_{\ge 2}$, 
when working in the groups~$G_k$. 
It will be clear from the context whether we mean $G$ or $G_k$.

We further write $z_{m,n}=[c_m,c_n]$ for every $m,n\in\N$, and we note the following descriptions for $H$ and $Z$: 
\[
H=\langle c_i\mid i\in \mathbb{N}\rangle, \qquad\text{and}\qquad  
Z=\langle c_l^{\,2},z_{m,n}\mid  l,m,n\in\mathbb{N}   \text{ with } n<m \rangle;
\]
also for every $i\ge 2$, 
$$
\gamma_i(G)\cap Z=\langle c_l^{\,2},z_{m,n}\mid l\ge i,\,\,1\le n<m,\,\, m+n\ge i \rangle.
$$

\smallskip

We recall a useful result from~\cite[Cor.~3.5]{HT}:
\begin{lemma}
\label{lemma p^k}
In the group~$G$, for $m,n\in \mathbb{N}$ we have
$$
[z_{m,n},x^{2^k}]=z_{m+{2^k},n}z_{m,n+2^k}z_{m+2^k,n+2^k}\quad \text{for every }k\in\mathbb{N}_0.
$$
\end{lemma}

Lastly, we recall terms of the lower 2-series and the dimension subgroup series of~$G_k$ from \cite[Prop.~3.10 and 3.11]{HT}:
For $k\in\N$, the length of the lower $2$-series of~$G_k$ is $2^{k+1}-1$ and
\begin{align*}
 P_0(G_k)&=G_k,\\
 P_1(G_k)&=\langle x_k^{\,2},y_k^{\,2}\rangle \gamma_2(G_k),
\end{align*}
and
$$
P_i(G_k)=
\begin{cases}
\langle x_k^{\,2^i}, c_i^{\,2}\rangle\gamma_{i+1}(G_k) &\text{for }2\le i\le 2^{k-1},\\
\langle x_k^{\,2^i}\rangle\gamma_{i+1}(G_k) &\text{for }2^{k-1}+1\le i\le 2^{k+1}-1.
\end{cases}
$$
For $k\in\N$, the length of the dimension subgroup series of $G_k$ is $2^{k+1}$ and
$$
D_i(G_k)=\langle x_k^{\,2^{l(i+1)}}\rangle\gamma_{\lceil \nicefrac{(i+1)}{2}\rceil}(G_k)^2\gamma_{i+1}(G_k) \quad \text{for } 0\le i\le 2^{k+1}-1,
$$
where $l(i+1)=\lceil\log_2 (i+1)\rceil$.

%%%%

\section{The \texorpdfstring{$2$}{2}-power series and the Frattini series of \texorpdfstring{$G$}{G}}\label{sec:normal-spectra}

Below we will precisely determine the terms $G^{2^k}\cap Z$  and $\Phi_k(G)$. For convenience, we recall the following standard commutator identities. 

\begin{lemma}
\label{lemma commutator identities}
Let $\Gamma=\langle a, b\rangle$ be a group, let $p$ be any prime, and let $r\in\N$.
For $u,v\in \Gamma$, let $K(u,v)$ denote the normal closure in $\Gamma$ of (i) all commutators in $\{u,v\}$ of weight at least $p^r$ that have weight at least $2$ in $v$, together with (ii) the $p^{r-s+1}$th
powers of all commutators in $\{u,v\}$ of weight less than $p^s$ and of weight at least $2$ in $v$ for $1\le s\le r$.
Then
\begin{align*}
 (ab)^{p^r}&\equiv_{K(a,b)}a^{p^r}b^{p^r}[b,a]^{\binom{p^r}{2}}[b,a,a]^{\binom{p^r}{3}}\cdots[b,a,
 \overset{p^r-2}{\ldots},a]^{\binom{p^r}{p^r-1}}[b,a,\overset{p^r-1}{\ldots},a],
 %\label{equ:commutator-formula-1}
 \\
 [a^{p^r},b]&\equiv_{K(a,[a,b])}[a,b]^{p^r}[a,b,a]^{\binom{p^r}{2}}\cdots[a,b,a,
 \overset{p^r-2}{\ldots},a]^{\binom{p^r}{p^r-1}}[a,b,a,\overset{p^r-1}{\ldots},a].
 %\label{equ:commutator-formula-2}
 \end{align*}
\end{lemma}

\medskip

Next, for $i,j,k\in\N$, we define the elements
$$
w_{i,j,k}=[z_{i+1,j},x,\overset{2^k-2}
\ldots,x][z_{i+2,j+1},x,\overset{2^k-3}
\ldots,x]\cdots [z_{i+2^k-2,j+2^k-3},x]z_{i+2^k-1,j+2^k-2},
$$
and  it follows by routine commutator calculus that 
\begin{equation}
\label{eq split}
[c_ic_j,x,\overset{2^k-1}{\ldots},x]=[c_i,x,\overset{2^k-1}{\ldots},x][c_j,x,\overset{2^k-1}{\ldots},x]w_{i,j,k}.
\end{equation}

 For $k\in \N$, we define the normal subgroup
$$
L_k=\left\langle w_{i,j,k},\, z_{m,n}\mid i,j,n\in\mathbb{N},\,m\ge 2^k
\right\rangle^G.
$$
The significance of this subgroup becomes clear in the next lemma, for which we need the following notation: for every $h\in H$, and referring to Lemma~\ref{lemma commutator identities}, we define $d_k(h)\in Z$ as
    $$
    (xh)^{2^k}=x^{2^k}[h,x,\overset{2^{k}-1}{\ldots},x]d_k(h).
    $$
Additionally, for all $k\in \mathbb{N}$ let
$$Q_k=\langle c_l^{\,2} \mid l\ge 2^{k-1}\rangle.
$$

\begin{lemma}
\label{lemma L_k}
In the pro-$2$ group $G$, for $k\in\N,$
\begin{enumerate}
    \item[(i)] for every $z\in Z$, we have
    $$
    [z,x,\overset{2^k-1}{\ldots},x]\in L_kQ_k;
    $$

    \item[(ii)] for $h_1,h_2\in H$, we have
    $$
    [h_1h_2,x,\overset{2^k-1}{\ldots},x]\equiv[h_1,x,{\overset{2^k-1}\ldots},x][h_2,x,\overset{2^k-1}{\ldots},x]\pmod{L_k};
    $$
    
    \item[(iii)] for every $h_1,h_2\in H$, we have
    $$
    d_k(h_1h_2)\equiv d_k(h_1)d_k(h_2)\pmod{L_k}.
    $$
\end{enumerate}
\end{lemma}

\begin{proof}
(i) Take first $m,n\in \mathbb{N}$ with $m>n$. Since $L_k$ is normal in $G$, we have $[w_{m-1,n,k},x]\in L_k$.
Note also that
\begin{equation}\label{eq:normal-w}
[w_{m-1,n,k},x]=[z_{m,n},x,\overset{2^{k}-1}{\ldots},x]w_{m,n+1,k}z_{m+2^k-1,n+2^k-1}^{-1},
\end{equation}
so we obtain $[z_{m,n},x,\overset{2^k-1}{\ldots},x]\in L_k$.

From~\eqref{eq split} we have
$$
[c_i^{\,2},x,\overset{2^{k}-1}{\ldots},x]=[c_i,x,\overset{2^{k}-1}{\ldots},x]^2w_{i,i,k}\in L_k Q_k
$$
for every $i\in\N$, so part (i) follows.

(ii) For every $h_1,h_2\in H$ write
\begin{align*}
&w(h_1,h_2)\\
&=[[h_1,x,h_2],x,\overset{2^k-2}
\ldots,x][[h_1,x,x,[h_2,x]],x,\overset{2^k-3}
\ldots,x]\cdots [h_1,x,\overset{2^k-1}{\ldots},x,[h_2,x,\overset{2^k-2}{\ldots},x]],
\end{align*}
so that $w(c_i,c_j)=w_{i,j,k}$.
It is easy to see that $w$ is bilinear
and that $w(z,h)=w(h,z)=1$ for every $h\in H$ and $z\in Z$.
Moreover, routine computations give
$$
[h_1h_2,x,\overset{2^{k}-1}{\ldots},x]=[h_1,x,\overset{2^{k}-1}{\ldots},x][h_2,x,\overset{2^{k}-1}{\ldots},x]w(h_1,h_2).
$$
Since every element in $H$ can be written in the form $c_{i_1}\cdots c_{i_n}z$ for some $n\in \N$ and $z\in Z$, the result follows.

(iii) From Lemma~\ref{lemma commutator identities} we easily deduce that for every $h\in H$ and $z\in Z$ we have %$d_k(h)\in Z$\comment{Maybe already said}, 
$d_k(z)=1$ and $d_k(hz)=d_k(h)$.
Now, for $h_1,h_2\in H$, we have
\begin{equation}
\label{eq: xh2h1}
\begin{split}
(xh_2h_1)^{2^k}&=(xh_2)^{2^k}[h_1,xh_2,\overset{2^k-1}{\ldots},xh_2]d_k(h_1)\\
&=x^{2^k}[h_2,x,\overset{2^k-1}{\ldots},x][h_1,xh_2,\overset{2^k-1}{\ldots},xh_2]d_k(h_1)d_k(h_2).
\end{split}
\end{equation}
Here, the first equality holds since $d_k(h_1)$ is the product of, on the one hand, the element $[h_1,x,\overset{2^{k-1}-1}{\ldots},x]^2$, and, on the other hand, commutators of weight~$2$ in~$h_1$, so writing $xh_2$ instead of~$x$ does not change the value of $d_k(h_1)$.
Routine computations give
$$
[h_1,xh_2,\overset{2^k-1}{\ldots},xh_2]=[h_1,x,\overset{2^k-1}{\ldots},x]\prod_{\ell=0}^{2^k-2} [h_1^{\,x},x,\overset{\ell}{\ldots},x,h_2,x,\overset{2^k-2-\ell}{\ldots},x].
$$
For $h_1,h_2\in H$ we write $\zeta(h_1,h_2)=\prod_{\ell=0}^{2^k-2} [h_1^{\,x},x,\overset{\ell}{\ldots},x,h_2,x,\overset{2^k-2-\ell}{\ldots},x]\in Z,
$ so that
$$
(xh_2h_1)^{2^k}=x^{2^k}[h_2,x,\overset{2^k-1}{\ldots},x][h_1,x,\overset{2^k-1}{\ldots},x]d_k(h_1)d_k(h_2)\zeta(h_1,h_2).
$$
This shows, by part (ii), that $d_k(h_1h_2)\equiv d_k(h_1)d_k(h_2)\zeta(h_1,h_2)\pmod{L_k}$, so
$$
\zeta(h_1,h_2)\equiv\zeta(h_2,h_1)\hspace{-6pt}\pmod{L_k},\ \ \text{ and }\ \ \zeta(h,h)\equiv 1\hspace{-6pt}\pmod{L_k}
$$
for every $h\in H$.
Next, notice that
\begin{equation}
\label{equation zeta(ci,cj)}
\begin{split}
[\zeta(c_i,c_j),x]
&=\prod_{\ell=0}^{2^k-2} [c_i^{\,x},x,\overset{\ell}{\ldots},x,c_j,x,\overset{2^k-1-\ell}{\ldots},x]\\
&=[c_i^{\,x},c_j,x,\overset{2^{k}-1}{\ldots},x]\zeta(c_{i+1},c_j)([c_i^{\,x},x,\overset{2^k-1}{\ldots},x,c_j])^{-1}\\
&\equiv\zeta(c_{i+1},c_j)\pmod{L_k}.\end{split}
\end{equation}
Since $\zeta(c_i,c_i)\in L_k$ and $\zeta(c_i,c_j)\equiv\zeta(c_j,c_i)\pmod{L_k}$, this shows that $\zeta(c_{i},c_j)\in L_k$ for every $i,j\in\N$.
This yields
$d_k(h_1h_2)\equiv d_k(h_1)d_k(h_2)\pmod{L_k}$.
\end{proof}

\begin{proposition}\label{prop:lower-bound-P}
For $k\in\N$,
the pro-2 group~$G$ satisfies
\begin{align*}
G^{2^k}\cap Z = L_kQ_k.
\end{align*}
\end{proposition}

\begin{proof}
For $r\in\mathbb{Z}$ and $h\in H$, we will be considering the expansion of $(x^rh)^{2^k}$, always according to Lemma~\ref{lemma commutator identities}.
Hence for conciseness, we will refrain from mentioning Lemma~\ref{lemma commutator identities} for the rest of the proof.

From 
\cite[Proof of Thm.~4.4]{HT} we know that $Q_k\gamma_{2^{k+1}}(G)\le G^{2^k}$. We now verify that $L_k\le G^{2^k}$. Indeed, first observe that, for $m\ge 2^k$ and $n\in\N$, 
\[
x^{-2^k}(xc_{m-2^k+1})^{2^k}\equiv c_m \pmod Z. 
\]
As $G^{2^k}$ is normal in~$G$, commutating with $c_n$ yields that we have $z_{m,n}\in G^{2^k}$.
Then we note that
\begin{equation}
\label{eq L_k normal closure}
L_k=\langle w_{i,i,k},z_{m,n}\mid i,n\in\N,\,m\ge 2^k\rangle^G
\end{equation}
since, $w_{i,i+1,k}=1$ and, as in Lemma~\ref{lemma p^k}, we have
$$
[w_{i,j,k},x]=w_{i+1,j,k}w_{i,j+1,k}w_{i+1,j+1,k}
$$
for every $i,j\in\N$. Hence it suffices to show that $w_{i,i,k}\in G^{2^k}$ for $i\in\N$. From~\eqref{eq split}, we have
\[
x^{-2^k}(xc_i^{\,2})^{2^k}=[c_i^{\,2},x,\overset{2^k-1}{\ldots},x]=c_{2^k+i-1}^{\,2}w_{i,i,k},
\]
from which it follows that $w_{i,i,k}\in G^{2^k}$, as required.

It remains to show that $G^{2^k}\cap Z \le L_kQ_k$.
Let $g=(x^rh)^{2^k}$, with $r\in\Z$ and $h\in H$.
We first show that $g$ is, modulo $L_kQ_k$, a product of elements of the form $(x^{2^m}h^*)^{s2^k}$ with  $m\in\N_0$, $s\in\Z$ and $h^*\in H$.
To do this, we show that actually every element of the form $((x^{2^n}h_1)^th_2)^{2^k}$, with $h_1,h_2\in H$, $n\in \N_0$ and $t$ odd, can be written in that way modulo $L_kQ_k$.
Let $N\in\N$ be such that $h_2\in\gamma_{N}(G)$.
If $N\ge 2^{k}$, then
$$
((x^{2^n}h_1)^th_2)^{2^k}=(x^{2^n}h_1)^{t2^k}[h_2,x,\overset{2^{k+n}-2^n}{\ldots},x]z,
$$
with $z\in\gamma_{2^{k+1}}(G)\cap Z\le L_kQ_k$.
 Write $h^*=[h_2,x,\overset{2^{k+n}-2^n-2^k+1}{\ldots},x]$. Since clearly $d_k(h^*)\in \gamma_{2^{k+1}}(G)\cap Z$, 
we have
$$
[h_2,x,\overset{2^{k+n}-2^n}{\ldots},x]=[h^*,x,\overset{2^k-1}{\ldots},x]\equiv x^{-2^k}(xh^*)^{2^k}\pmod{L_kQ_k}.
$$
Suppose now that $N<2^{k}$ and that the assertion follows whenever $h_2$ lies in $\gamma_{N+1}(G)$.
Let $t^*$ be such that $tt^*\equiv 1\pmod 4$.
Then,
$$
((x^{2^n}h_1)^th_2)^{2^k}=((x^{2^n}h_1)^t(h_2^{\,t^*})^t)^{2^k}=((x^{2^n}h_1h_2^{\,t^*})^th_3)^{2^k},
$$
with $h_3\in\gamma_{N+1}(G)$, and the assertion follows by reverse induction.

Therefore, since $G$ is the closure of the subgroup consisting of  elements of the form $x^rh$ with $r\in\Z$ and $h\in H$, then $G^{2^k}$ is the closure of the subgroup generated by elements of the form
$$
g=(xh_{1})^{\epsilon_12^k}\cdots(xh_{n})^{\epsilon_n2^k}(x^{2^{m_1}}h^*_{1})^{\epsilon_{n+1}2^k}\cdots(x^{2^{m_{\nu}}}h^*_{\nu})^{\epsilon_{n+\nu}2^k}
$$
with $n,\nu\in\mathbb{N}_0$, $m_j\in\N$, $h_i,h^*_j\in H$ and $\epsilon_i\in\{-1,1\}$.
Now, working modulo~$L_kQ_k$ in the equivalences below, and noting that $\gamma_{2^{k+1}-2}(G)\cap Z\le L_kQ_k$,  Lemma~\ref{lemma L_k} gives
\begin{align*}
 g&=\prod_{i=1}^{n}(xh_i)^{\epsilon_i2^k}\prod_{i=1}^{\nu} (x^{2^{m_i}}h^*_i)^{\epsilon_{n+i}2^k}\\
 &\equiv
\prod_{i=1}^{n}(x^{2^k}[h_i,x,\overset{2^k-1}{\ldots},x]d_k(h_i))^{\epsilon_i}\!\prod_{i=1}^{\nu} ( x^{2^{k+m_{i}}}[h_i^*,x,\overset{2^{k+m_i}-2^{m_i}}{\ldots},x])^{\epsilon_{n+i}} \\%\!\!\!\pmod{L_kQ_k}\\
 &\equiv
\prod_{i=1}^{n}x^{\epsilon_i2^k}[h_i^{\epsilon_i},x,\overset{2^k-1}{\ldots},x]d_k(h_i^{\epsilon_i})\!\prod_{i=1}^{\nu} x^{\epsilon_{n+i}2^{k+m_{i}}}[h_i^*,x,\overset{2^{k+m_i}-2^{m_i}}{\ldots},x]^{\epsilon_{n+i}}\cdot [\widehat{h},x,\overset{2^k-1}{\ldots},x]\\%\!\!\!\pmod{L_kQ_k}\\
&\equiv
x^{r2^k}[hh^*,x,\overset{2^k-1}{\ldots},x]d_k(h)\pmod{L_kQ_k},
\end{align*}
with %$z\in\gamma_{2^{k+1}-1}(G)\cap Z\le L_k$, 
$r\in\Z$, $h\in H$ and $\widehat{h},h^*\in\gamma_{2^{k}}(G)$.
Since $d_k(h)\in Z$, it follows that $g\in Z$ if and only if $r=0$ and $hh^*\in Z$.
Hence $h\in \gamma_{2^k}(G)Z$, which implies that $d_k(h)\in L_k$. Further, from Lemma~\ref{lemma L_k}(i), we have $[hh^*,x,\overset{2^k-1}{\ldots},x]\in L_kQ_k$.
Therefore %we have $d_k(h)=1$ and $[h,x,\overset{2^k}{\ldots},x]\in L_k$, so that 
$g\equiv 1\pmod{L_kQ_k}$, as desired.
\end{proof}

The next observation will be useful for the next section. 
Recall that for a filtration series~$\mathcal{S}:G=S_0\ge S_1\ge\cdots$ of~$G$, we denote by $\mathcal{S}\mid_Z$ the restriction $\mathcal{S}\cap Z:Z=S_0\cap Z\ge S_1\cap Z\ge\cdots$.

\begin{corollary}\label{cor:M-P}
For the pro-2 group $G$ and $K\le_{\mathrm{c}} Z$, we have
\[
\hdim_Z^{\mathcal{P}\mid_Z}(K)=\hdim_Z^{\mathcal{M}\mid_Z}(K).
\]
\end{corollary}

\begin{proof}
For $k\in\mathbb{N}$, it follows from the construction of $G_k$ that 
\[
M_k\cap Z=\langle c_l^{\,2}, z_{m,n} \mid l>2^k,\,m> 2^k,\,n\in\mathbb{N}\rangle 
\le G^{2^k}\cap Z;
\]
compare also \cite[Lem.~3.3(ii)]{HT}. 
We  claim that
\begin{align*}
L_kQ_k&=\langle w_{i,i,k},\,[z_{j+1,j},x,\overset{2^{k}-1}{\ldots},x],z_{2^k,n}, c_l^{\,2}\mid 1\le i\le 2^{k-1},\\
&\qquad\qquad\quad 1\le j\le 2^{k-1}-1,\,1\le n\le 2^k-1,\, 2^{k-1}\le l \le 2^k\rangle (M_k\cap Z).
\end{align*}
Indeed, recall from \eqref{eq L_k normal closure} that
$$
L_k=\langle w_{i,i,k},z_{m,n}\mid i,n\in\N,\,m\ge 2^k\rangle^G,
$$
and moreover, as seen in~\eqref{eq:normal-w} we have
\begin{align*}
[w_{i,i,k},x]
&\equiv [z_{i+1,i},x,\overset{2^k-1}{\ldots},x]w_{i+1,i+1,k}\pmod{\gamma_{2^{k+1}}(G)\cap Z},
\end{align*}
so the claim follows.

Now, by counting the number of generators, we have
$$
\log_2 |L_kQ_k:M_k\cap Z|\le 2^{k+1}+2^{k-1}.
$$
Hence, by Proposition~\ref{prop:lower-bound-P},
\begin{align*}
\lim_{k\to\infty}\frac{\log_2|G^{2^k}\cap Z :M_k\cap Z |}{\log_2|Z:G^{2^k}\cap Z|}&\le\lim_{k\to\infty}\frac{2^{k+1}+2^{k-1}}{\log_2|Z:G^{2^k}\cap Z|}=0,
\end{align*}
where the equality follows from the fact that $\log_2|Z:G^{2^k}\cap Z|\ge 2\binom{2^{k-1}}{2}$; cf. \cite[Proof of Thm.~4.4]{HT}. The result now follows from \cite[Lem.~2.2]{KTZR}.
\end{proof}

%%%%

We now determine the Frattini series~$\mathcal{F}$ precisely.
Recall that for $i,j,\ell\in\N$, 
$$
w_{i,j,\ell}=[z_{i+1,j},x,\overset{2^\ell-2}
\ldots,x][z_{i+2,j+1},x,\overset{2^\ell-3}
\ldots,x]\cdots [z_{i+2^\ell-2,j+2^\ell-3},x]z_{i+2^\ell-1,j+2^\ell-2}.
$$
For $j,\ell\in\N$, we write $w_{j,\ell}:=w_{j,j,\ell}$.
%\[
%w_{j,\ell}:=w_{j,j,\ell}= [z_{j+1,j},x,\overset{2^{\ell}-1}\ldots,x][z_{j+2,j+1},x,\overset{2^{\ell}-2}\ldots,x]\cdots  [z_{j+2^{\ell}-1,j+2^{\ell}-2},x]z_{j+2^{\ell},j+2^{\ell}-1}.
%\]
Then for $k\in\N$, we define
\begin{align*}
\Psi_k&=\langle [w_{j,\ell-1} ,x^{2^{\ell}}, x^{2^{\ell+1}},\ldots, x^{2^{k-1}}]\mid 2\le \ell\le k, \,2^{\ell-2}\le j\le 2^{\ell-1}-1 \rangle\le\gamma_{2^k}(G),\\
\Lambda_k&=\langle [z_{m,n},x^{2^{\ell}},x^{2^{\ell+1}},\ldots, x^{2^{k-1}}]\mid 2\le \ell\le k-1, \,2^{\ell-1}\le n<m<2^{\ell} \rangle\le\gamma_{2^k+1}(G),\\
\Theta_k&=\langle z_{m,n},\,z_{m',n'}\mid m,n\ge 2^{k-1},\,m'\ge 2^k,\,n'\in\N\rangle\ge \gamma_{2^k+2^{k-1}}(G)\cap \langle z_{i,j}\mid i,j\in\N\rangle.
\end{align*}
Note that the $z_{m,n}$ in the presentation of $\Theta_k$ are precisely the elements that we would obtain if we let $\ell=k$ in the presentation of $\Lambda_k$.

\begin{proposition}
\label{lemma Frattini explicit}
For each $k\in\N$, we have
$$
\Phi_k(G)=\langle x^{2^k},c_j\mid j\ge 2^k\rangle Q_k\Psi_k\Lambda_k\Theta_k,
$$
with $\log_2|\Psi_k|\le 2^{k-1}-1$ and, for $k\ge 2$,
\begin{align*}
\log_2|\Lambda_k\Theta_k:\Theta_k|&=%\sum_{\ell=2}^{k-1}2^{2\ell-4}=\frac{2^{2k-4}-1}{3}.
\sum_{\ell=2}^{k-1}(2^{\ell-1}-1)2^{\ell-2}=\frac{2^{2k-3}+1}{3} - 2^{k-2}.
\end{align*}
\end{proposition}

\begin{proof}
We proceed by induction on $k$, with the result for $k=1$ being trivial.
Thus, assume the result true for $k-1$.
As $\Phi_k(G)=\Phi_{k-1}(G)^2$ and $[\Phi_{k-1}(G),\Phi_{k-1}(G)]\le\Phi_{k-1}(G)^2$, 
it is clear that
$$
\Phi_k(G)=\langle x^{2^k}\rangle Q_k[\Phi_{k-1}(G),\Phi_{k-1}(G)].
$$
From \cite[Proof of Thm.~4.5]{HT} we have
$$
\langle c_j\mid j\ge 2^k\rangle\le\Phi_k(G)\ \ \text{and}\ \ c_{2^k-1}\not\in\Phi_k(G),
$$
and moreover, as $\Phi_k(G)\trianglelefteq G$, it follows that $z_{m,n}\in\Phi_k(G)$ for all $m\ge 2^k$ and $n\in\N$.
Also, since $c_i\in\Phi_{k-1}(G)$ for all $i\ge 2^{k-1}$, we have $z_{m,n}\in\Phi_k(G)$ for all $m,n\ge 2^{k-1}$.
This shows that $\Theta_k\le\Phi_k(G)$.
Thus, noting that $\langle c_j\mid j\ge 2^k\rangle Q_k\Theta_k$ is normal in~$G$, we claim that
\begin{equation}\label{eq:equivalence}
[\Phi_{k-1}(G),\Phi_{k-1}(G)]\equiv\Psi_k\Lambda_k\pmod{\langle c_j\mid j\ge 2^k\rangle Q_k\Theta_k}.
\end{equation}
Note that $[\Psi_k\Lambda_k,x^{2^{k-1}}]\le \gamma_{2^k+2^{k-1}}(G)\cap \langle z_{i,j}\mid i,j\in\N\rangle\le \Theta_k$, so $\Psi_k\Lambda_k$ is normal in $\Phi_{k-1}(G)$ modulo~$\Theta_k$. Here for a subgroup $A\le G$, by $[A,x^{2^{k-1}}]$ we mean $\langle [a,x^{2^{k-1}}] \mid a\in A\rangle$.
 Thus, in order to prove \eqref{eq:equivalence}, let us compute the commutators coming from $[Q_{k-1},x^{2^{k-1}}]$, $[\Psi_{k-1},x^{2^{k-1}}]$, $[\Lambda_{k-1},x^{2^{k-1}}]$, $[\Theta_{k-1},x^{2^{k-1}}]$, and $[c_j,x^{2^{k-1}}]$ for $j\ge 2^{k-1}$.

It is routine to see that 
\begin{equation}\label{eq:generator-basis}
[\Lambda_{k-1}\Theta_{k-1},x^{2^{k-1}}]\equiv\Lambda_k\pmod{\Theta_k}
\end{equation}
and that $[\Psi_{k-1},x^{2^{k-1}}]\subseteq\Psi_k$.
Observe also that for $j\ge 2^{k-1}$ we have $[c_j,x^{2^{k-1}}]=c_{j+2^{k-2}}^{\,2}c_{j+2^{k-1}}z$ with $z\in \Theta_k$, $c_{j+2^{k-2}}^{\,2}\in Q_k$ and $c_{j+2^{k-1}}\in\langle c_l\mid l\ge 2^k\rangle$.
Hence, it suffices to show that $\Psi_k\equiv [Q_{k-1}\Psi_{k-1}
,x^{2^{k-1}}]\pmod{Q_k\Theta_k}$.
For this purpose, as seen in the proof of Proposition~\ref{prop:lower-bound-P}, notice that $[c_j^{\,2},x^{2^{k-1}}]\equiv w_{j,k-1} \pmod {Q_k}$, which lies in $\Psi_k$ if $2^{k-2}\le j\le 2^{k-1}-1$ and in $\Theta_k$ if $j\ge 2^{k-1}$.
Moreover, this gives us precisely the generators of $\Psi_k$ that are not in $[\Psi_{k-1},x^{2^{k-1}}]$, so the assertion follows.

Next, since $\Psi_k=[\Psi_{k-1},x^{2^{k-1}}]\langle w_{j,k-1}\mid 2^{k-2}\le j\le 2^{k-1}-1\rangle$, the inductive hypothesis yields
$$
|\Psi_k|\le |\Psi_{k-1}|+2^{k-2}\le 2^{k-2}-1+2^{k-2}=2^{k-1}-1,
$$
as desired.

For the final statement, we show that the generators in the presentation of $\Lambda_k$ generate it independently modulo~$\Theta_k$.
Then, the result will follow just by counting the generators.
In order to do so, define for every $s\in \N$ the subgroup
$$
\Theta_{k,s}=\langle z_{m,n},z_{m',n'}\mid m,n\ge s, \,m'\ge 2^k\rangle\Theta_k.
$$
Clearly, $\Theta_{k,s}=\Theta_k$ for every $s\ge 2^{k-1}$.
Note by Lemma~\ref{lemma p^k} that the generators $[z_{m,n},x^{2^\ell},x^{2^{\ell+1}},\ldots,x^{2^{k-1}}]$ lie in $\Theta_{k,s}$ for every $1\le s\le n<m$ while
\begin{equation}
\label{presentation}
[z_{m,n},x^{2^\ell},x^{2^{\ell+1}},\ldots,x^{2^{k-1}}]\equiv z_{m+2^{\ell}+\cdots+2^{k-1},n}\pmod{\Theta_{k,n+1}}.
\end{equation}
Now, observe that $[z_{3,2},x^{2^2},x^{2^3},\ldots,x^{2^{k-1}}]$ is the unique generator in the presentation of $\Lambda_k$ with $n=2$.
This is, hence, the unique generator not lying in $\Theta_{k,3}$, and hence it is independent from all the others.
By induction, suppose that   for all $n<s$ with $s>3$, the generators $[z_{m,n},x^{2^\ell},x^{2^{\ell+1}},\ldots,x^{2^{k-1}}]$ of~$\Lambda_k$ are independent.
Take then the generators with $n=s$ and note by (\ref{presentation}) that they are all independent modulo~$\Theta_{k,s+1}$.
They are hence also independent from all the other previously considered generators, and the assertion follows.
\end{proof}

%%%

\section{The finitely generated Hausdorff spectra of \texorpdfstring{$G$}{G}}\label{sec:fg-spec}

In this section, we prove the main result for the case $p=2$. 
As the proof is quite technical, we outline the general approach here. First we perform three reduction steps. One, we observe that, for a closed subgroup~$K$ of~$G$, the Hausdorff dimension of~$K$ in~$G$ is unchanged by adding or removing a finite number of generators in~$Z$ to~$K$; see Lemma~\ref{lemma remove generators}. Two, for a finitely generated closed subgroup~$K$ of~$G$, we observe that $\hdim_G^{\mathcal{S}}(K)=\hdim_Z^{\mathcal{S}|_Z}(K\cap Z)$ and hence it suffices to work within the subgroup~$Z$;  see Corollary~\ref{cor:reduce-to-Z}.
Three, with the second observation, we make use of the fact that the restrictions to~$Z$ of the filtration series~$\mathcal{P}$ and $\mathcal{M}$  are essentially the same (see  Corollary~\ref{cor:M-P}), hence 
it suffices to consider the four filtration series $\mathcal{L}$, $\mathcal{D}$, $\mathcal{M}$ and $\mathcal{F}$. As mentioned in the introduction, in the final part of the proof, which boils down to a counting argument, we introduce the idea of counting in blocks, which simplifies computations.

\smallskip

\begin{lemma}
\label{lemma remove generators}
Let $L$ be a closed subgroup of $G$ and $\mathcal{S} \in\{\mathcal{L},\mathcal{D},\mathcal{P},\mathcal{F},\mathcal{M}\}$.
Then, for every $n\in\N$ and $z_1,\ldots,z_n\in Z$, we have
$$
\hdim_G^{\mathcal{S}}(L)=\hdim_G^{\mathcal{S}}(L\langle z_1,\ldots,z_n\rangle^G ).
$$
\end{lemma}

\begin{proof}
 Write $K=L\langle z_1,\ldots,z_n\rangle^G $ and $A=\langle z_1,\ldots,z_n\rangle^G$, and note that $K= LA$.  Now, recalling that we write $\mathcal{S}:G=S_0\ge S_1\ge\cdots$ for a filtration series of~$G$, we have
\[
|LAS_i :S_i|\le |LS_i:S_i|\cdot |AS_i:S_i|
\]
and by~\cite[Lem.~5.3]{KTZR},
\[
\hdim_G^{\mathcal{S}}(A)=\hdim_G^{\mathcal{S}}(Z)\cdot \hdim_Z^{\mathcal{S}\mid_Z}(A).
\]
Thus, using the proof of
~\cite[Lem.~2.3]{HK}, it follows  that $\hdim_G^{\mathcal{S}}(A)=0$ has strong Hausdorff dimension. Therefore
\begin{equation*}
\hdim_G^{\mathcal{S}}(K)= \hdim_G^{\mathcal{S}}(LA) \le \hdim_G^{\mathcal{S}}(L)+\hdim_G^{\mathcal{S}}(A)\le \hdim_G^{\mathcal{S}}(L),
\end{equation*}
and since $L\le K$, equality follows.
\end{proof}

Let $\mathcal{S}:G=S_0\ge S_1\ge\cdots$ be a filtration series of~$G$. For $k\in\N$, we define
$$
n_k=\min\{i\in\N\mid\gamma_i(G)\le S_k\}.
$$
From~\cite{HT} we know that if $\mathcal{S}\in\{\mathcal{L},\mathcal{D}\}$ then $n_k=k+1$, if $\mathcal{S}=\mathcal{M}$ then $n_k= 2^{k+1}$, if $\mathcal{S}=\mathcal{P}$ then $n_k\le 2^{k+1}$,
and if $\mathcal{S}=\mathcal{F}$ then $n_k\le 2^k+2^{k-1}-1$. Further we have  that
\begin{equation}
\label{eq division}
\lim_{k\rightarrow\infty}\dfrac{n_k}{\log_2|Z:S_k\cap Z|}=0
\end{equation}
for every $\mathcal{S}\in\{\mathcal{L},\mathcal{D},\mathcal{P},\mathcal{F},\mathcal{M}\}$.

On the other hand, for every $k\in\N$ we also define
$$
\alpha_k=\min\{i\in\N\mid x^{2^i}\in S_k\}.
$$
It is easy to see that if $\mathcal{S}\in\{\mathcal{L},\mathcal{P},\mathcal{F},\mathcal{M}\}$ then we have $\alpha_k=k$, and if $\mathcal{S}=\mathcal{D}$ then  $\alpha_k=\lceil\log_2(k+1)\rceil$.

\smallskip

This next lemma is key to obtaining the second reduction step mentioned above.
\begin{lemma}
\label{lemma hasudorff dimension}
Let $\mathcal{S}\in\{\mathcal{L},\mathcal{D},\mathcal{P},\mathcal{M},\mathcal{F}\}$ and let $K=\langle x^{2^l}h, h_1,\ldots,h_d\rangle$ be a finitely generated closed subgroup of~$G$, for some $l, d\in\mathbb{N}_0$ and $h,h_1,\ldots, h_d\in H$. Then,
\[
\lim_{i\to \infty} \dfrac{\log_2|KS_i \cap Z : (K\cap Z)(S_i\cap Z)|}{\log_2|Z:S_i\cap Z|}=0.
\]
\end{lemma}

\begin{proof}
For every $i\in\N$, let $n_i$ and $\alpha_i$ be defined as above. For $\mathcal{S}\in\{\mathcal{L},\mathcal{D},\mathcal{P},\mathcal{M},\mathcal{F}\}$, it follows from the previous sections that $S_i=\langle x^{2^{\alpha_i}},S_{i,H}\rangle$
with $S_{i,H}=S_i\cap H$.
Note also that $\alpha_i\xrightarrow{i\rightarrow\infty}\infty$. It suffices to show that for every $i$ such that $\alpha_i\ge l$ we have
$$
|KS_i \cap Z : (K\cap Z)(S_i\cap Z)|\le 2^{(2d+2)n_i+1}.
$$
Indeed,  by (\ref{eq division}) we then have
$$
\lim_{i\to \infty} \dfrac{\log_2|KS_i \cap Z : (K\cap Z)(S_i\cap Z)|}{\log_2|Z:S_i\cap Z|}\le\lim_{i\to \infty}\dfrac{(2d+2)n_i+1}{\log_2|Z:S_i\cap Z|}=0,
$$
as desired.

So we fix such an $i$ and write for simplicity  $S_H=S_{i,H}$. In the following, we will be working modulo $(K\cap Z)(S_i\cap Z)$, hence for simplicity we set $(K\cap Z)(S_i\cap Z)=1$. We have
$$
KS_i=\langle x^{2^l}h,h_1,\ldots,h_d,x^{2^{\alpha_i}}\rangle S_H.
$$
Observe, however, that
\begin{equation}\label{eq:one-x}
\langle x^{2^{l}}h, x^{2^{\alpha_i}}\rangle =\langle x^{2^{l}}h, h_0
%\overline{h} 
\rangle,
\end{equation}
where $h_0=x^{2^{\alpha_i}}(x^{2^{l}}h)^{-{2^{\alpha_i-l}}}\in H$, and it can be seen, by Lemma~\ref{lemma commutator identities} and since $\alpha_i-l\ge 0$, that $h_0
\in \gamma_{2^{\alpha_i}-2^l+1}(G)Z$. 
So
$$
KS_i=\langle x^{2^{l}}h,h_0,h_1,\ldots,h_d\rangle S_H.
$$
Now, for every $0\le n\le d$ and every $m\in\N_0$ define $h_{n,m}=[h_n,x^{2^l}h,\overset{m}{\ldots},x^{2^l}h]$, and let
$$
K_H=\langle h_{n,m}\mid 0\le n\le d,\, m\in\N_0\rangle.
$$
Clearly we have $KS_i=\langle x^{2^l}h\rangle  K_HS_H$, and observe that $K_H$ is normalised by~$x^{2^l}h$.
Therefore, as $S_H$ is normal in $G$, every element~$g$ of~$KS_i$ can be written as
$$
g=(x^{2^l}h)^{\beta}ks
$$
with $\beta\in\Z_2$, $k\in K_H$ and $s\in S_H$.
Moreover, if $g\in KS_i\cap Z$ then we must have $\beta=0$, so that $g=ks$, which we will assume for the rest of the proof.

We proceed by considering two cases: when one of $k,s$ is in~$Z$, and when both $k,s\notin Z$. If $k\in Z$, then $s=k^{-1}g\in S_i\cap Z=1$, and so $g=k\in K_H\cap Z$.
Similarly, if $s\in Z$ then $s=1$ and so $g=k\in K_H\cap Z$.
We claim that
\[
K_H\cap Z=\langle h_{0,m}^{\,2}, [h_0, h_{n,m}]\mid m\in \mathbb{N}_0,\,1\le n\le d\rangle.
\]
First, observe that $c_{2^{\alpha_i}+1}\in S_iZ$.
Indeed, this is clear for $\mathcal{S}\in\{\mathcal{L},\mathcal{D}\}$;
for $\mathcal{S}=\mathcal{P}$ we have $\alpha_i=i$ and $(xy)^{2^{i}}\equiv x^{2^i}y^{2^i} c_{2^{i}}\pmod Z$, so in particular $c_{2^i+1}\in S_iZ$, and, since $G^{2^i}\le \Phi_i(G)$, we also have $c_{2^{i}+1}\in S_iZ$ if  $\mathcal{S}=\mathcal{F}$;
and for $\mathcal{S}=\mathcal{M}$ this follows from \cite[Prop.~2.6(1)]{KT}.
In particular, this implies that $\gamma_{2^{\alpha_i}+1}(G)\le S_iZ\cap H$.

Thus, since Lemma \ref{lemma commutator identities} yields $h_{0,m}\in \gamma_{2^{\alpha_i}+1}(G)Z$ for every $m\in\N$, it follows that
$$
[h_{0,m},H]\le [\gamma_{2^{\alpha_i}+1}(G)Z,H]\le [S_iZ\cap H,H]\le S_i\cap Z=1.
$$
Note also that $[h_{n,m},h_{r,s}]\in K\cap Z=1$ for all $1\le n,r\le d$ and $m,s\in \N_0$. Finally, we have $h_{n,m}^{\,2}\in K\cap Z=1$ for all $1\le n\le d$ and $m\in\N_0$, so the claim follows.

% On the one hand, we have $h_{n,m}^2\in K\cap Z=1$ for all $1\le n\le d$ and $m\in\N_0$.
% On the other hand, let us see that $h_{0,m}\in C_G(H)$ for all $m\in\N$.
% Note that $h_{0,m}\in\gamma_{2^{\alpha_i}+1}$.
% If $\mathcal{S}\in\{\mathcal{L},\mathcal{D}\}$, then we have $\gamma_{2^{\alpha}}(G)\le\gamma_{n_i}(G)\le S_i$, and so $[h_{0,m},H]\le \gamma_{2^{\alpha_i}+1}(G)\cap Z\le S_i\cap Z=1$;
% For the case $\mathcal{S}=\mathcal{P}$, it follows from~\eqref{eq:commutators-in-S}.
% If $\mathcal{S}=\mathcal{M}$ then we have $\gamma_{2^{\alpha_i}+1}(G)\le ZS_{i}$ by~\cite[Prop.~2.6(1)]{KT}, and it is clear that $[ZS_i,H]\le S_i\cap H=1$.
% Finally, since $[h_{n,m},h_{r,s}]\in K\cap Z=1$ for all $1\le n,r\le d$ and $m,s\in \N_0$, the claim follows.

Now, since $\gamma_{n_i}(G)\cap Z\le S_i\cap Z=1$, we have
$$
|\langle h_{0,m}^{\,2}, [h_0, h_{n,m}]\mid m\in\N_0, \, 1\le n\le d \rangle|\le 2^{(d+1)n_i}.
$$

Suppose now $k,s\not\in Z$. If there exists another element $t\in S_H$ such that $kt\in Z$, then
$$
s^{-1}t=(ks)^{-1}kt\in S_i\cap Z=1,
$$
so $s=t$. Thus, for each $k\in K_H$ there exists at most one element $z$ in $Z$ such that $ks=z$ for some $s\in S_H$.
In particular, this shows that there are at most $|K_H|$ elements $g=ks$ in $Z$ such that $k,s\not\in Z$.
Now, since the nilpotency class of $K_H$ is~$2$, it is easy to see that $K_H\cap Z=\Phi(K_H)$, and so
$$
|K_H|=|K_H:\Phi(K_H)||\Phi(K_H)|\le 2^{(d+1)n_i}2^{(d+1)n_i}=2^{(2d+2)n_i}.
$$

Finally, summing up, we get that
$$
|KS_i\cap Z|\le 2^{(d+1)n_i}+2^{(2d+2)n_i}\le 2^{(2d+2)n_i+1},
$$
as desired.
\end{proof}

\begin{corollary}\label{cor:reduce-to-Z}
Let $\mathcal{S}\in\{\mathcal{L},\mathcal{D},\mathcal{P},\mathcal{M},\mathcal{F}\}$ and let $K=\langle x^{2^l}h, h_1,\ldots,h_d\rangle$ be a finitely generated closed subgroup of~$G$, for some $l, d\in\mathbb{N}_0$ and $h,h_1,\ldots, h_d\in H$. Then
\[
\hdim_G^{\mathcal{S}}(K)=\hdim_Z^{\mathcal{S}|_Z}(K\cap Z).
\]
\end{corollary}

\begin{proof}
Recall, from the proof of~\cite[Thm.~2.10]{KT}, that $KZ/Z$ has strong Hausdorff dimension in~$W$. Hence, as $Z$ has strong Hausdorff dimension in~$G$, it follows from \cite[Lem.~2.2]{KT} that 
\begin{equation*}
\hdim_G^{\mathcal{S}}(K)= \varliminf_{k \to \infty} \frac{\log_2
                             \lvert KS_k \cap Z : S_k \cap Z \rvert}{\log_2
                             \lvert Z : S_k \cap Z \rvert},
\end{equation*}
where $\mathcal{S}$ is given by $\mathcal{S}\colon G = S_0 \ge S_1\ge \cdots$. Moreover, by Lemma~\ref{lemma hasudorff dimension} we have
\begin{equation*}%\label{eq:fg-formula}
\begin{split}
   &\varliminf_{k \to \infty} \frac{\log_2|KS_k \cap Z : S_k \cap Z|}{\log_2|Z : S_k \cap Z|}\\
    &\quad\hspace{-10pt}=\lim_{k \to \infty}\frac{\log_2|KS_k \cap Z :(K\cap Z)(S_k\cap Z)|}{\log_2|Z : S_k \cap Z|}+\varliminf_{k \to \infty}\dfrac{\log_2|(K\cap Z)(S_k\cap Z):S_k \cap Z|}{\log_2|Z : S_k \cap Z|}\\
    &\quad\hspace{-10pt}=\varliminf_{k \to \infty}\frac{\log_2|(K\cap Z)(S_k\cap Z):S_k \cap Z|}{\log_2|Z : S_k \cap Z|},
\end{split}
\end{equation*}
and thus $\hdim_G^{\mathcal{S}}(K)=\hdim_Z^{\mathcal{S}|_Z}(K\cap Z)$, as required.
\end{proof}

For a subgroup $K=\langle x^{2^l}h, h_1,\ldots,h_d\rangle$ of $G$, and, in view of Corollary \ref{cor:reduce-to-Z}, for the purpose of computing $K\cap Z$, we will fix the notation used in the proof of Lemma \ref{lemma hasudorff dimension}.
Thus, for every $1\le n\le d$ and every $m\in\N_0$, we write $h_{n,m}=[h_n,x^{2^l}h,\overset{m}{\ldots},x^{2^l}h]$. Suppose  $i_1,\ldots, i_d\in \mathbb{N}$ are such that
\begin{align*}
    h_1\in \gamma_{i_1}(G)Z\backslash \gamma_{i_1+1}(G)Z,\,\,\ldots\,\,,
    h_d\in \gamma_{i_d}(G)Z\backslash \gamma_{i_d+1}(G)Z.
\end{align*}
Notice by Lemma~\ref{lemma commutator identities} that  $h_{n,m}\equiv c_{i_n+m2^l}\pmod{\gamma_{i_n+m2^l+1}(G)Z}.$
Also, we set $I=\{i_n+m2^l\mid 1\le n\le d,\, m\in\N_0\}$, and for every $j\in I$, we write $c^*_j=h_{n,m}$, where $n,m$ are such that $j=i_n+m2^l$.
If $j\in\N\backslash I$, then we just write $c^*_j=c_j$.
Finally, for every $j,k\in\N$ define $z_{j,k}^*=[c_j^*,c_k^*]$ and consider the set
$$
B_Z=\{(c_i^*)^2,z_{j,k}^*\mid i,j,k\in\N\}.
$$
Thus, since $c_j^*\equiv c_j\pmod{\gamma_{j+1}(G)Z}$ for every $j\in \N$, it follows that
\begin{equation}\label{eq:basis-change}
z_{j,k}^*\equiv z_{j,k}\pmod{\gamma_{j+k+1}(G)\cap \langle z_{m,n}\mid m\ge j,\,n\ge k\rangle}.%Z_{(\ell +1)}},
\end{equation}
%where $\ell=\min\{j,k\}$.
Hence, the set~$B_Z$ is a basis for~$Z$.

Let now
$$
K_Z=\langle(c_i^*)^2,z^*_{j,k}\mid i,j,k\in I\rangle.
$$
It is clear that $K=\langle x^{2^l}h\rangle K_H$ where, as in Lemma~\ref{lemma hasudorff dimension},
$$
K_H=\langle h_{n,m}\mid 1\le n\le d,\, m\in\N_0\rangle.
$$
Since $x^{2^l}h$ normalises $K_H$, it follows that a general element $g\in K$ can be written as
$$
g=(x^{2^{l}}h)^{\alpha}c^*_{j_1}\cdots c^*_{j_n}z,
$$
with $\alpha\in\Z_2$, $n\in\N_0$, $j_r<j_{r+1}$ for every $1\le r\le n-1$ and $z\in K_Z$.
Now, if $g\in Z$, then we must have $\alpha=0$ and $n=0$, and so $K\cap Z=K_Z$.

\

Now that we explicitly know the subgroup $K\cap Z$, the strategy that we will follow for computing its Hausdorff dimension in $Z$ will be based on counting \emph{blocks}.
For a fixed $l\in\N_0$ and for every $r,s\in\N_0$ with $s\le r$, we define a \emph{block} in~$Z$ as
$$
B_{r,s}=\langle z^*_{i+r2^l,j+s2^l}\mid 1\le i,j\le 2^l\rangle.
$$

\begin{lemma}
\label{lemma blocks}
Fix $l,d\in\N_0$, let $\mathcal{S}\in\{\mathcal{L},\mathcal{D},\mathcal{M},\mathcal{F}\}$ and let $K=\langle x^{2^l}h, h_1,\ldots,h_d\rangle$ be a subgroup of~$G$ as defined above.
Write $m(k)=\min\{j\mid c_j^{\,2}\in S_k\}$ and let
$$
\Delta_k=
\begin{cases}
\langle B_{r,s}\mid B_{r,s}\cap S_k\neq 1\rangle\langle (c_j^*)^2\mid j\ge m(k)\rangle & \text{if }\ \mathcal{S}\in\{\mathcal{L},\mathcal{D},\mathcal{M}\},\\
\langle B_{r,s}\mid B_{r,s}\cap \Theta_k\neq 1 \rangle\langle (c^*_j)^2\mid j\ge m(k)\rangle \Psi_k \Lambda_k & \text{if }\ \mathcal{S}=\mathcal{F}.
\end{cases}
$$
Then, for the filtration series $\mathcal{S}^*:Z=\Delta_0\ge\Delta_1\ge\cdots$ of~$Z$, we have 
$$
\hdim_{Z}^{\mathcal{S}\mid_Z}(K\cap Z)=\hdim_Z^{\mathcal{S}^*}(K\cap Z).
$$
\end{lemma}

\begin{proof}
Note that
\begin{align*}
\Delta_k=
&\begin{cases}
\langle B_{r,s}\mid B_{r,s}\cap \gamma_{k+1}(G)\neq 1 \rangle\langle (c_j^*)^2\mid j\ge k\rangle & \text{if }\mathcal{S}=\mathcal{L},\\
\langle B_{r,s}\mid B_{r,s}\cap \gamma_{k+1}(G)\neq 1 \rangle\langle (c_j^*)^2\mid j\ge \lceil (k+1)/2\rceil\rangle & \text{if }\mathcal{S}=\mathcal{D},\\
\langle B_{r,s}\mid r\ge 2^{k-l}\rangle\langle (c^*_j)^2\mid j\ge 2^k+1\rangle & \text{if }\mathcal{S}=\mathcal{M},
\end{cases}
\end{align*}
and  $B_Z$ being a basis for~$Z$, we can express the generators of $\gamma_{k+1}(G)\cap Z$ and $\Theta_k$ in terms of the elements in~$B_Z$.
More precisely, one can easily see that
\begin{equation}\label{eq:gamma_k+1}
\gamma_{k+1}(G)\cap Z=\langle (c^*_l)^2,z^*_{m,n}\mid l\ge k+1,\,\,1\le n<m,\,\, m+n\ge k+1 \rangle
\end{equation}
and
\begin{equation}\label{eq:Theta_k}
    \Theta_k=\langle z^*_{m,n},\,z^*_{m',n'}\mid m,n\ge 2^{k-1},\,m'\ge 2^k,\,n'\in\N\rangle.
\end{equation}
Therefore, we clearly have $S_k\cap Z\le\Delta_k$.
Consider now the following family of blocks:
\begin{equation*}
    \begin{split}
        \mathcal{B}
        & = \begin{cases}
        \{B_{r,s}\mid B_{r,s}\le\Delta_k,\, B_{r,s}\not\le S_k \text{ and } s\le r\} & \text{if }\ \mathcal{S}\in\{\mathcal{L},\mathcal{D},\mathcal{M}\},\\
        \{B_{r,s}\mid B_{r,s}\le\Delta_k,\, B_{r,s}\not\le \Theta_k \text{ and } s\le r\} & \text{if }\ \mathcal{S}=\mathcal{F}.
        \end{cases}
    \end{split}
\end{equation*}

If $\mathcal{S}\in\{\mathcal{L},\mathcal{D}\}$, then $B_{r,s}\in\mathcal{B}$ only if $(r+s)2^l\le k-2$ and $(r+s+2)2^l\ge k+1$, 
so $|\mathcal{B}|\le 2\lfloor (k+1)/2^{l+1}\rfloor+1$.
For $\mathcal{S} =\mathcal{M}$, we have $|\mathcal{B}|= 0$ and for $\mathcal{S} =\mathcal{F}$, we have $|\mathcal{B}|=2^{k-l}-1$.

Notice also that $\log_2|B_{r,s}|\le 2^{2l}$, hence, if $\mathcal{S}\in\{\mathcal{L},\mathcal{D},\mathcal{M}\}$, then we have
\begin{equation*}
    \begin{split}
        \lim_{k\rightarrow\infty}\frac{\log_2|\Delta_k:S_k\cap Z|}{\log_2|Z:\Delta_k|}
        \le
        \lim_{k\rightarrow\infty}\frac{|\mathcal{B}|2^{2l}}{\log_2|Z:\gamma_{k+1-2^{l+1}}(G)\cap Z|}
        = 0,
    \end{split}
\end{equation*}
while if $\mathcal{S}=\mathcal{F}$, we have
\begin{equation*}
    \begin{split}
        \lim_{k\rightarrow\infty}\frac{\log_2|\Delta_k:S_k\cap Z|}{\log_2|Z:\Delta_k|}
        \le
        \lim_{k\rightarrow\infty}\frac{|\mathcal{B}|2^{2l}}{\log_2|Z:\gamma_{2^k-2^{l+1}}(G)\cap Z|}
        = 0
    \end{split}
\end{equation*}
by \cite[Proof of Thm.~4.5]{HT}.

Now, \cite[Lem.~2.2]{KTZR} yields the result.
\end{proof}

\begin{proof}[Proof of Theorem~\ref{thm:main} (For the case $p=2$)]
Let $K\le G$ be a finitely generated closed subgroup of~$G$. We observe that if $K\le H$, then $K$ is finite and hence $\hdim_G^{\mathcal{S}}(K)=0$. Therefore we suppose that $K\not \le H$.

Without loss of generality we may assume, as done in (\ref{eq:one-x}), that $K$ has only one generator of the form $x^{2^l}h$ with $l\in\N_0$ and $h\in H$. In addition, by Lemma~\ref{lemma remove generators}, we may further assume that
$$
K=\langle x^{2^l}h,h_1,\ldots, h_d \rangle
$$
for some $d\in \mathbb{N}_0$ and $h_1,\ldots,h_d\in H\backslash Z$, and by Corollary~\ref{cor:reduce-to-Z}, we have that $\hdim_G^{\mathcal{S}}(K)=\hdim_Z^{\mathcal{S}|_Z}(K\cap Z)$.
Also, following the notation introduced before Lemma \ref{lemma blocks}, we have
$$
K\cap Z=\langle(c_i^*)^2,z^*_{j,k}\mid i,j,k\in I\rangle.
$$

%Recall, from the proof of~\cite[Thm.~2.10]{KT}, that $KZ/Z$ has strong Hausdorff dimension in~$W$. Hence, as $Z$ has strong Hausdorff dimension in~$G$, it follows from \cite[Lem.~2.2]{KT} that 
%\begin{equation*}
%\hdim_G^{\mathcal{S}}(K)= \varliminf_{k \to \infty} \frac{\log_2
%                             \lvert KS_k \cap Z : S_k \cap Z \rvert}{\log_2
%                             \lvert Z : S_k \cap Z \rvert},
%\end{equation*}
%where $\mathcal{S}$ is given by $\mathcal{S}\colon G = S_0 \ge S_1\ge \cdots$. Moreover, by Lemma~\ref{lemma hasudorff dimension} we have
%\begin{equation}\label{eq:fg-formula}
%\begin{split}
%   &\varliminf_{k \to \infty} \frac{\log_2|KS_k \cap Z : S_k \cap Z|}{\log_2|Z : S_k \cap Z|}\\
%    &\quad\hspace{-10pt}=\lim_{k \to \infty}\frac{\log_2|KS_k \cap Z :(K\cap Z)(S_k\cap Z)|}{\log_2|Z : S_k \cap Z|}+\varliminf_{k \to \infty}\dfrac{\log_2|(K\cap Z)(S_k\cap Z):S_k \cap Z|}{\log_2|Z : S_k \cap Z|}\\
    %&\quad\hspace{-10pt}=\varliminf_{k \to \infty}\frac{\log_2|(K\cap Z)(S_k\cap Z):S_k \cap Z|}{\log_2|Z : S_k \cap Z|},
%\end{split}
%\end{equation}

Let $i_1,\ldots, i_d\in \mathbb{N}$ be such that
\begin{align*}
    h_1\in \gamma_{i_1}(G)Z\backslash \gamma_{i_1+1}(G)Z,\,\,\ldots\,\,,
    h_d\in \gamma_{i_d}(G)Z\backslash \gamma_{i_d+1}(G)Z,
\end{align*}
where, in the spirit of~(\ref{eq:one-x}), we may assume that $i_1<\cdots < i_d$ and further that $i_1,\ldots,i_d$ are pairwise non-equivalent modulo $2^l$. Indeed, suppose without loss of generality that $i_d\equiv i_r \pmod {2^l}$ for some $1\le r\le d-1$. As mentioned before, we have
$$
[h_n,x^{2^l}h,\overset{m}{\ldots},x^{2^l}h]\equiv c_{i_n+m2^l}\pmod{\gamma_{i_n+m2^l+1}(G)Z}
$$
for all $n\in\{1,\ldots,d\}$ and $m\in\N_0$. In particular
$$
[h_r,x^{2^l}h,\overset{m}{\ldots},x^{2^l}h]\equiv h_d\pmod{\gamma_{i_d+1}(G)Z}
$$
for some $m> 0$, and so repeating the standard cancelling process as in~(\ref{eq:one-x}), either, after several steps,
\begin{enumerate}
\item[$\bullet$] the process stops if it gives a generator $\tilde{h}\in Z$ (in this case, we can ignore this generator by Lemma~\ref{lemma remove generators});

\item[$\bullet$] the process stops if it gives a generator $\tilde{h}\in\gamma_j(G)Z\backslash\gamma_{j+1}(G)Z$ where $j\not \equiv i_1,\ldots,i_{d-1} \pmod {2^l}$; or

\item[$\bullet$] at each cancelling step, the new generator $\tilde{h}\in\gamma_j(G)Z\backslash\gamma_{j+1}(G)Z$ satisfies $j\equiv i_r \pmod {2^l}$, for some $r\in\{1,\ldots,d-1\}$. 
\end{enumerate}
In the third case, %in light of Corollary~\ref{cor:reduce-to-Z}\comment{Why do we need Corollary~\ref{cor:reduce-to-Z}?},
we may ignore the generator~$\tilde{h}$, since each cancellation process replaces the generator $\tilde{h}\in\gamma_j(G)Z\backslash\gamma_{j+1}(G)Z$ with a generator~$\hat{h}\in\gamma_\nu(G)Z\backslash\gamma_{\nu+1}(G)Z$, where $\nu>j$. In other words, we can replace the generator $\tilde{h}$ with a generator~$\hat{h}$ in $\gamma_{\mu}(G)Z$, with $\mu$ arbitrary large, hence in each finite quotient $\frac{(K\cap Z)(S_k\cap Z)}{ S_k \cap Z}$, the images of the generators of $K\cap Z$ that involve~$\hat{h}$ can be assumed to be trivial.

Next, for $r\in\mathbb{N}$ and setting $Z_{(r)}=\langle z_{m,n}\mid m,n\ge r\rangle$, it can be proved, as done in \cite[Sec.~4]{HT}, that $Z_{(r)}$ has strong Hausdorff dimension~$1$ in~$G$ with respect to~$\mathcal{S}$.
Thus, it is not difficult to see that we can also assume that $i_d\le 2^l$. 
Indeed, suppose that $\lambda2^l< i_d \le(\lambda+1)2^l$ for some $\lambda\in\N$.
It turns out that the generators of $K\cap Z$ of the form~$(c^*_j)^{\,2}$, for $j\in\N$, are insignificant in the computation of $\hdim_Z^{\mathcal{S}|_Z}(K\cap Z)$, hence a direct computation shows that $\hdim_Z^{\mathcal{S}|_Z}(K\cap Z)=\hdim_Z^{\mathcal{S}|_Z}(K\cap Z_{(\lambda2^l)})$. This then equals the Hausdorff dimension of $K\cap Z_{(\lambda2^l)}$ in~$Z_{(\lambda2^l)}$ by \cite[Lem.~5.3]{KTZR}.

Now, from Corollary~\ref{cor:M-P}, it suffices to consider the four filtration series $\mathcal{L}$, $\mathcal{D}$, $\mathcal{M}$ and $\mathcal{F}$.
Moreover, if $\mathcal{S}^*:Z=\Delta_0\ge\Delta_1\ge\cdots$ is as in Lemma \ref{lemma blocks}, then we have $\hdim_Z^{\mathcal{S}\mid_Z}(K\cap Z)=\hdim_Z^{\mathcal{S}^*}(K\cap Z)$.
Define the following disjoint families of blocks:
\begin{equation*}
    \begin{split}
        \mathcal{B}_1
        & =\begin{cases}
        \{B_{r,s}\mid B_{r,s}\cap\Delta_k=1 \text{ and } s<r\}
        & \text{if }\ \mathcal{S}\in\{\mathcal{L},\mathcal{D},\mathcal{M}\},\\
        \{B_{r,s}\mid B_{r,s}\cap\Theta_k=1 \text{ and } s<r\}
        & \text{if }\ \mathcal{S}=\mathcal{F},
        \end{cases}\\
        \mathcal{B}_2
        & =\begin{cases}
        \{B_{r,s}\mid B_{r,s}\cap\Delta_k=1 \text{ and } s=r\}
        & \text{if }\ \mathcal{S}\in\{\mathcal{L},\mathcal{D},\mathcal{M}\},\\
        \{B_{r,s}\mid B_{r,s}\cap\Theta_k=1 \text{ and } s=r\}
        & \text{if }\ \mathcal{S}=\mathcal{F}.
        \end{cases}
    \end{split}
\end{equation*}
For $\mathcal{S}\in\{\mathcal{L},\mathcal{D}\}$, it is routine to see that $|\mathcal{B}_2|\le \lfloor (k+1)/2^{l+1}\rfloor$ and that there exists $q\in\mathbb{Q}$ such that $\lim_{k\rightarrow\infty}|\mathcal{B}_1|/(k+1)^2=q$.
Similarly for $\mathcal{S} =\mathcal{M}$, we have
$|\mathcal{B}_2|= 2^{k-l}$, and there exists $q\in\mathbb{Q}$ such that $\lim_{k\rightarrow\infty}|\mathcal{B}_1|/2^{2k}=q$.
For $\mathcal{S}=\mathcal{F}$, we have $|\mathcal{B}_2|=2^{k-1-l}-1$ and $\lim_{k\rightarrow\infty}|\mathcal{B}_1|/2^{2k}=q$ for some $q\in\Q$.

Note that if $s<r$, then $\log_2|B_{r,s}|=2^{2l}$ and $\log_2|B_{r,s}\cap K|=d^2$,
while otherwise, if $s=r$,
then $\log_2|B_{r,s}|=2^2l-1-2^{l-1}$ and $\log_2|B_{r,s}\cap K|:=t\le d^2$.
With this, we are now ready to compute the Hausdorff dimension of $K$ with respect to $\mathcal{L}$, $\mathcal{D}$ and $\mathcal{M}$; compare (\ref{eq: final}).
Let us focus first, however, on the case where $\mathcal{S}=\mathcal{F}$.
From Proposition~\ref{lemma Frattini explicit} we note that
$\frac{1}{2^{2k}} \log_2\lvert\Psi_k\rvert$
is negligible as $k$ approaches infinity, hence we can ignore~$\Psi_k$.
As done in the proof of Lemma \ref{lemma blocks}, we can also express the generators of  $\Lambda_k$ in terms of the elements in~$B_Z$. More precisely, and referring to~\eqref{eq:generator-basis}, we have $\Lambda_k^*\Theta_k=\Lambda_k\Theta_k$, where
\begin{equation}\label{eq:Lambda_k}
\Lambda^*_k=\langle [z^*_{m,n},x^{2^{i}},x^{2^{i+1}},\ldots, x^{2^{k-1}}]\mid 2\le i\le k-1, \,2^{i-1}\le n<m<2^{i} \rangle.
\end{equation}
Observing  that
\[
\Lambda^*_k\cap K\ge\langle [z_{m,n}^*,x^{2^i},x^{2^{i+1}},\ldots, x^{2^{k-1}}]\mid m,n\in I,\,l\le i\le k-1, \,2^{i-1}\le n<m<2^i \rangle
\]
(where in the above, we emphasise that $i\ge l$)
and that for $2\le i\le k-1$ we have $[z_{m,n}^*,x^{2^i},x^{2^{i+1}},\ldots, x^{2^{k-1}}]\not\in\Lambda^*_k\cap K$ if $m\not\in I$ or $n\not\in I$, it follows that
$$
\log_2|\Lambda^*_k\cap K:\langle [z_{m,n}^*,x^{2^i},\ldots, x^{2^{k-1}}]\mid m,n\in I,\,l\le i\le k-1, \,2^{i-1}\le n<m<2^i \rangle|%\overset{i\rightarrow\infty}{\longrightarrow}0.
$$
approaches $0$ as $k\rightarrow\infty$. 
Writing
\[
    U_1:=\log_2|(\Lambda^*_k\cap K)\Theta_k:\Theta_k|\quad\text{and}\quad
    U_2:=\log_2| \Lambda^*_k\Theta_k:\Theta_k|,
\]
and recalling from Proposition~\ref{lemma Frattini explicit} that the generators  in the presentation of~$\Lambda^*_k$ in~\eqref{eq:Lambda_k} generate it independently modulo~$\Theta_k$, 
%referring to Proposition~\ref{lemma Frattini explicit}, 
we see that $U_1/U_2$ approaches $d^2/2^{2l}$ as $k\rightarrow \infty$.

Finally, removing the assumption that $\mathcal{S}=\mathcal{F}$, we obtain
\begin{equation}
\label{eq: final}
    \begin{split}
        \hdim_Z^{\mathcal{S}^*}(K\cap Z) &=\lim_{k\rightarrow\infty}\frac{\log_2|(K\cap Z)\Delta_k:\Delta_k|}{\log_2|Z:\Delta_k|}\\
        &=
        \lim_{k\rightarrow\infty}\frac
        {|\mathcal{B}_1|d^2+|\mathcal{B}_2|t +\log_2|\langle (c_j^*)^2\mid j\in I, \,j< m(k)\rangle|-\delta_{\mathcal{F}}U_1}
        {|\mathcal{B}_1|2^{2l}+|\mathcal{B}_2|(2^{2l}-2^{l-1}) +\log_2|\langle (c_j^*)^2\mid  j< m(k)\rangle|-\delta_{\mathcal{F}}U_2}\\
        &=
        \lim_{k\rightarrow\infty}\frac
        {|\mathcal{B}_1|d^2 - \delta_{\mathcal{F}}U_1}
        {|\mathcal{B}_1|2^{2l} - \delta_{\mathcal{F}}U_2}\\
        &=d^2/2^{2l},
    \end{split}
\end{equation}
for every $\mathcal{S}\in\{\mathcal{L},\mathcal{D},\mathcal{M},\mathcal{F}\}$, where 
\[
\delta_{\mathcal{F}}=\begin{cases}
1 & \text{if }\mathcal{S}=\mathcal{F},\\
0 & \text{otherwise}.
\end{cases}
\]

As $K$ was arbitrary, it follows that
\[
\hspec^{\mathcal{S}}_{\text{fg}}(G)=\{d^2/2^{2l}\mid  l\in\N,\, 0\le d\le 2^l\}.\qedhere
\]
\end{proof}

\section{The pro-\texorpdfstring{$p$}{p} groups \texorpdfstring{$\mathfrak{G}(p)$}{G(p)} for odd primes}
\label{sec:iterated}

For convenience, we write $\mathfrak{G}=\mathfrak{G}(p)$, for an odd prime~$p$.
In this section we determine the terms of the six filtration series of~$\mathfrak{G}$, or their respective intersections with~$Z$.
Then the proof of Theorem \ref{thm:main} will follow exactly in the same way as for $p=2$.

As the terms of the filtration series $\mathcal{L}$, $\mathcal{D}$ and $\mathcal{M}$ can be found in~\cite{HK}, it remains to settle the filtration series $\mathcal{P}$, $\mathcal{F}$, and $\mathcal{I}$.

We begin by clarifying the terms of the series $\mathfrak{G}^{p^k}\cap Z$. In principle, the analysis of the subgroups $\mathfrak{G}^{p^k}\cap Z$ is analogous to the $p=2$ case, however certain differences arise, due to the fact that $c_i^{\,2}\notin Z$ for $i\in\N$.

Using the same notation for the case $p=2$, for $i,j,k\in\N$, we let
$$
w_{i,j,k}=[z_{i+1,j},x,\overset{p^k-2}
\ldots,x][z_{i+2,j+1},x,\overset{p^k-3}
\ldots,x]\cdots [z_{i+p^k-2,j+p^k-3},x]z_{i+p^k-1,j+p^k-2},
$$
and for $k\in\N$,
$$
L_k=\left\langle w_{i,j,k},\,d_k(y),\, z_{m,n}\mid i,j,n\in\N 
,\,m\ge p^k
\right\rangle^{\mathfrak{G}},
$$
 where
 $d_k(h)\in Z$ is defined as before, that is,
    $$
    (xh)^{p^k}=x^{p^k}[h,x,\overset{p^k-1}{\ldots},x]d_k(h)
    $$
for $h\in H$.

\begin{lemma}
\label{lemma L_k odd}
In the pro-$p$ group~$\mathfrak{G}$, for $k\in\N,$
\begin{enumerate}
    \item[(i)] for every $z\in Z$, we have
    $$
    [z,x,\overset{p^k-1}{\ldots},x]\in L_k;%Q_k;
    $$

    \item[(ii)] for $h_1,h_2\in H$, we have
    $$
    [h_1h_2,x,\overset{p^k-1}{\ldots},x]\equiv[h_1,x,{\overset{p^k-1}\ldots},x][h_2,x,\overset{p^k-1}{\ldots},x]\pmod{L_k};
    $$
    
    \item[(iii)] for every $h\in H$, we have $d_k(h)\in L_k$.%for every $h_1,h_2\in H$, we have
   % $$
   % d_k(h_1h_2)\equiv d_k(h_1)d_k(h_2)\pmod{L_k}.
    %$$
\end{enumerate}
\end{lemma}

\begin{proof}
(i) and (ii): This is just as in Lemma \ref{lemma L_k}.
% , where for (i) we note that $Z=\langle z_{m,n}\mid m,n\in\N\rangle$.

(iii) For $h_1,h_2\in H$, as before let  $\zeta(h_1,h_2)=\prod_{\ell=0}^{p^k-2} [h_1^{\,x},x,\overset{\ell}{\ldots},x,h_2,x,\overset{p^k-2-\ell}{\ldots},x]\in Z$. 
Proceeding as in the $p=2$ case,  we obtain
$d_k(h_1h_2)\equiv d_k(h_1)d_k(h_2)\zeta(h_1,h_2)\pmod{L_k}$, and
$\zeta(h_1,h_2)\equiv\zeta(h_2,h_1)\pmod{L_k}$.
 Next, notice that for every $i,j\in\N$ we have
$$
\zeta(c_i,c_j)^x=\zeta(c_i^{\,x},c_j^{\,x})=\zeta(c_ic_{i+1},c_jc_{j+1})=\zeta(c_i,c_j)\zeta(c_{i+1},c_j)\zeta(c_i,c_{j+1})\zeta(c_{i+1},c_{j+1})
 $$
 and from the equivalence
 $[\zeta(c_i,c_j),x]
\equiv\zeta(c_{i+1},c_j)\pmod{L_k}$ in (\ref{equation zeta(ci,cj)}),
 we obtain  $$ \zeta(c_i,c_{j+1})\equiv \zeta(c_{i+1},c_{j+1})\pmod {L_k}.
$$% for every $i,j\in\N$.
Therefore, as $\zeta(c_{p^k},c_{j})\in L_k$, it follows that $\zeta(c_i,c_j)\in L_k$ whenever $(i,j)\neq (1,1)$.
Now for $h\in \gamma_2(\mathfrak{G})$, we observe that 
\begin{align*}
\zeta(h,h)&\equiv d_k(h^2) d_k(h)^{-2} \pmod {L_k},
\end{align*}
and as $d_k(h^2)=d_k(h)^4$, we conclude that $d_k(h)\in L_k$.
Since $d(y)\in L_k$ by definition, the result follows.
\end{proof}

With this, the corresponding statements for Proposition \ref{prop:lower-bound-P} and Corollary~\ref{cor:M-P} are obtained:

\begin{proposition}\label{prop:lower-bound-P-odd}
For $k\in\N$,
the pro-$p$ group~$\mathfrak{G}$ satisfies
\begin{align*}
\langle z_{m,n}\mid m\ge p^k, \, n\in N\rangle\le \mathfrak{G}^{p^k}\cap Z \le L_k.
\end{align*}
\end{proposition}

\begin{corollary}\label{cor:M-P p odd}
For the pro-$p$ group~$\mathfrak{G}$ and $K\le_{\mathrm{c}} Z$, we have
\[
\hdim_Z^{\mathcal{P}\mid_Z}(K)=\hdim_Z^{\mathcal{M}\mid_Z}(K).
\]
\end{corollary}

\medskip

Next, in analogue to the even prime case, one can determine the Frattini series~$\mathcal{F}$ precisely, which we include here for completeness. Following the notation in~\cite{HK}, we write $[i]_p=\frac{p^i-1}{p-1}$ for $i\in \N_0$. For
 $k\in\mathbb{N}$   we correspondingly have %define
\begin{align*}
\Lambda_k&=
\langle [z_{m,n},x^{p^{i}},x^{p^{i+1}},\ldots, x^{p^{k-1}}]\mid 2\le i\le k-1,\, m,n\ge 1+[i]_p \rangle,\\
\Theta_k&=\langle z_{m,n},\,z_{m',n'}\mid m,n\ge 1+[k-1]_p ,\,m'\ge 1+[k]_p ,\,n'\in\N\rangle.
\end{align*}
Observe that $\gamma_{1+2[k-1]_p+p^{k-1}}(\mathfrak{G})\le \Theta_k$.

\begin{proposition}
For each $k\in\mathbb{N}$, we have
\begin{align*}
\Phi_k(\mathfrak{G})=&\langle x^{p^k},
c_j \mid 
j\ge 1+[k]_p\rangle \Lambda_k\Theta_k.
\end{align*}
\end{proposition}

\medskip

Finally, we consider the iterated $p$-power series~$\mathcal{I}$ for $\mathfrak{G}$.
This involves a fusion of the techniques used for the $p$-power series and the Frattini series.

For $i,j,k\in\N$, with $i,j\ge p^{k-1}$, we define the elements
\begin{align*}
&\widetilde{w}_{i,j,k}=[z_{i+1,j},x,\overset{p^k-p^{k-1}-1}
\ldots,x]
[z_{i+2,j+1},x,\overset{p^k-p^{k-1}-2}
\ldots,x]
\cdots z_{i+p^{k}-p^{k-1},j+p^k-p^{k-1}-1},
\end{align*}
and for every $h\in \gamma_{p^{k-1}}(\mathfrak{G})$, we define $\widetilde{d}_k(h)\in Z$ as
$$
(x^{p^{k-1}}h)^p=x^{p^k}[h,x,\overset{p^k-p^{k-1}}{\ldots},x]\widetilde{d}_k(h).
$$
Let thus
$$
\widetilde{L}_k=
\langle 
\widetilde{w}_{i,j,k}, \widetilde{d}_k(y) \mid i,j\ge p^{k-1}
\rangle^{\mathfrak{G}},
$$
and, for $k\in\N$, we define 
\begin{align*}
\widetilde{\Theta}_k  &=\langle z_{m,n} \mid m\ge p^k,\, n\in\N\rangle,\\
\widetilde{\Lambda}_k &=\big\langle [z_{m,n},x^{p^{k-1}},\overset{p-1}\ldots, x^{p^{k-1}}] \mid (p^{k-1}\le m< p^k,\,   m-\big\lfloor \tfrac{m}{p^{k-1}}\big\rfloor p^{k-1} \le n< p^{k-1})\\
&\qquad\qquad\quad\qquad\qquad \vee\quad (m\in\{ 2p^{k-1}, 3p^{k-1},\ldots, (p-2)p^{k-1}\}, \,\, n=p^{k-1} )\big\rangle,\\
\widetilde{\Omega}_k &=\big\langle [z,x^{p^{k-1}},\overset{p-1}\ldots, x^{p^{k-1}}] \mid z\in \widetilde{\Omega}_{k-1}\cup \widetilde{\Lambda}_{k-1}\big\rangle,\\
\widetilde{\Psi}_k  &=\big\langle [z,x^{p^{k-1}},\overset{p-1}\ldots, x^{p^{k-1}}] \mid z\in \widetilde{\Psi}_{k-1}\cup\widetilde{L}_{k-1}\big\rangle,
\end{align*}
where $\widetilde{\Omega}_1=\widetilde{\Lambda}_1$ and $\widetilde{\Psi}_1=\widetilde{L}_1$.
These definitions, together with the following lemma, will play the role that Lemma~\ref{lemma L_k}(i) had for the filtration series~$\mathcal{P}$.

\begin{lemma}\label{lem:combinatorial-proof}
For $k\in\N$, we have
% \[
% \big\langle [z_{m,n},x^{p^{k-1}},\overset{p-1}\ldots, x^{p^{k-1}}] \mid p^{k-1}\le m<p^k,\, n<m \big\rangle\widetilde{\Theta}_k=\widetilde{\Lambda}_k\widetilde{\Theta}_k
% \]
% and\comment{Isn't it the same?}
\[
[\widetilde{\Theta}_{k-1},x^{p^{k-1}},\overset{p-1}{\ldots},x^{p^{k-1}}]\widetilde{\Theta}_k=\widetilde{\Lambda}_k\widetilde{\Theta}_k.
\]
%Moreover, the generators given in the presentation of $\widetilde{\Lambda}_k$ form a minimal generating set for $\tilde{\Lambda}_k$ modulo $\widetilde{\Theta}_k$.
\end{lemma}

The proof of the above lemma, though not difficult, is a rather long combinatorial argument. Therefore we defer the proof to  the appendix. Moreover, for the purpose of computing Hausdorff dimensions, the set of generators given in the presentation of $\widetilde{\Lambda}_k$ is close enough to a minimal generating set for $\widetilde{\Lambda}_k$ modulo $\widetilde{\Theta}_k$; see the appendix for further details.

Below we have the analogues of Lemma~\ref{lemma L_k}(ii) and (iii).

\begin{lemma}
\label{lemma tilde L_k}
In the pro-$p$ group~$\mathfrak{G}$, for $k\in\N,$
\begin{enumerate}
    \item[(i)] for $h_1,h_2\in \gamma_{p^{k-1}}(\mathfrak{G})$, we have
    $$
    [h_1h_2,x,\overset{p^k-p^{k-1}}{\ldots},x]\equiv[h_1,x,{\overset{p^k-p^{k-1}}\ldots},x][h_2,x,\overset{p^k-p^{k-1}}{\ldots},x]\pmod{\widetilde{L}_k};
    $$
    \item[(ii)] for every $h\in \gamma_{p^{k-1}}(\mathfrak{G})$, we have $\widetilde{d}_k(h)\in \widetilde{\Theta}_k\widetilde{\Lambda}_k\widetilde{L}_k$.
\end{enumerate}
\end{lemma}

\begin{proof}
Part (i) follows as in Lemma \ref{lemma L_k}.
For part (ii), let $d^*_k(h)$ and $d^{**}_k(h)$ be such that
$$
(x^{p^{k-1}}h)^p=x^{p^k}[h,x^{p^{k-1}},\overset{p-1}{\ldots},x^{p^{k-1}}]d^*_k(h)
$$
and
$$[h,x^{p^{k-1}},\overset{p-1}{\ldots},x^{p^{k-1}}]=[h,x,\overset{p^k-p^{k-1}}{\ldots},x]d^{**}_k(h),
$$
and note that $\widetilde{d}_k(h)=d^*_k(h)d^{**}_k(h)$.
Now, similar to (\ref{eq: xh2h1}), for $h_1,h_2\in\gamma_{p^{k-1}}(\mathfrak{G})$ we have
\begin{equation*}
\begin{split}
(x^{p^{k-1}}h_2h_1)^p&=(x^{p^{k-1}}h_2)^{p}[h_1,x^{p^{k-1}}h_2,\overset{p-1}{\ldots},x^{p^{k-1}}h_2]d^*_k(h_1)\\
&=x^{p^k}[h_2,x,\overset{p^k-p^{k-1}}{\ldots},x][h_1,x^{p^{k-1}}h_2,\overset{p-1}{\ldots},x^{p^{k-1}}h_2]d^*_k(h_1)\widetilde{d}_k(h_2),
\end{split}
\end{equation*}
and routine computations give
\begin{equation*}
    \begin{split}
        [h_1,x^{p^{k-1}}h_2,&\overset{p-1}{\ldots},x^{p^{k-1}}h_2]\\
        &=[h_1,x,\overset{p^k-p^{k-1}}{\ldots},x]
        \prod_{\ell=0}^{p-2} [h_1^{\,x^{p^{k-1}}},x,\overset{\ell p^{k-1}}{\ldots},x,h_2,x,\overset{(p-2-\ell)p^{k-1}}{\ldots},x]d^{**}_k(h_1).
    \end{split}
\end{equation*}
Writing $\widetilde{\zeta}(h_1,h_2)=\prod_{\ell=0}^{p-2} [h_1^{\,x^{p^{k-1}}},x,\overset{\ell p^{k-1}}{\ldots},x,h_2,x,\overset{(p-2-\ell)p^{k-1}}{\ldots},x]\in Z$, we have
$$
(x^{p^{k-1}}h_2h_1)^{p}=x^{p^k}[h_2,x,\overset{p^k-p^{k-1}}{\ldots},x][h_1,x,\overset{p^k-p^{k-1}}{\ldots},x]\widetilde{d}_k(h_1)\widetilde{d}_k(h_2)\widetilde{\zeta}(h_1,h_2).
$$
Everything now follows as in the proof of Lemma~\ref{lemma L_k odd}.
\end{proof}

\begin{proposition}\label{prop:lower-bound-I}
For $k\in\N$, the pro-$p$ group~$\mathfrak{G}$ satisfies
\begin{align*}
 \widetilde{\Theta}_k\widetilde{\Lambda}_k\widetilde{\Omega}_k
 \le I_k(\mathfrak{G})\cap Z \le \widetilde{\Theta}_k\widetilde{\Lambda}_k\widetilde{\Omega}_k\widetilde{L}_k\widetilde{\Psi}_k.
\end{align*}
\end{proposition}

\begin{proof}
Using Lemmata \ref{lem:combinatorial-proof} and \ref{lemma tilde L_k}, we can adjust the proof of Proposition~\ref{prop:lower-bound-P} to the filtration series~$\mathcal{I}$.
The unique significant difference arises when, as done at the end of the aforementioned proof, we consider the element $[hh^*,x,\overset{p^{k}-p^{k-1}}{\ldots},x]$, where in our case $h\in\gamma_{p^{k-1}}(\mathfrak{G})\cap I_{k-1}(\mathfrak{G})$, $h^*\in\gamma_{2p^{k-1}}(\mathfrak{G})$, and $hh^*\in Z$.

First note that $\gamma_{2p^{k-1}}(\mathfrak{G})\le I_{k-1}(\mathfrak{G})$ by \cite[Prop. 3.4]{HK}, and so
$$
hh^*\in I_{k-1}(\mathfrak{G})\cap Z\le \widetilde{\Theta}_{k-1}\widetilde{\Lambda}_{k-1}\widetilde{\Omega}_{k-1}\widetilde{L}_{k-1}\widetilde{\Psi}_{k-1}
$$
by the inductive hypothesis.
Now, by definition, we have $[\widetilde{\Lambda}_{k-1}\widetilde{\Omega}_{k-1},x,\overset{p^{k}-p^{k-1}}{\ldots},x]\le \widetilde{\Omega}_k$ and  $[\widetilde{L}_{k-1}\widetilde{\Psi}_{k-1},x,\overset{p^{k}-p^{k-1}}{\ldots},x]\le \widetilde{\Psi}_k$.
Also, Lemma~\ref{lem:combinatorial-proof} yields $[\widetilde{\Theta}_{k-1},x,\overset{p^{k}-p^{k-1}}{\ldots},x]\le \widetilde{\Lambda}_k\widetilde{\Theta_k}$, so the result follows.
\end{proof}

\begin{corollary}
For the pro-$p$ group~$\mathfrak{G}$ and $K\le_{\mathrm{c}} Z$, we have
\[
\hdim_Z^{\mathcal{I}\mid_Z}(K)=\hdim_Z^{\mathcal{{\widetilde{\mathcal{S}}}}}(K),
\]
where $\widetilde{\mathcal{S}}$ is the filtration of $Z$ defined by the subgroups $\widetilde{\Theta}_k\widetilde{\Lambda}_k\widetilde{\Omega}_k$.
\end{corollary}

\medskip

Finally, the proofs of Lemma~\ref{lemma remove generators}, Lemma~\ref{lemma hasudorff dimension}, Corollary~\ref{cor:reduce-to-Z} and Lemma~\ref{lemma blocks}, and the $p=2$ proof of Theorem~\ref{thm:main}  work exactly in the same way for $p$ odd and for $\mathcal{L}$, $\mathcal{D}$,  $\mathcal{M}$, $\mathcal{P}$, $\mathcal{I}$ 
and $\mathcal{F}$ (actually it is easier, as $H^p=1$). Hence Theorem~\ref{thm:main} follows.

\appendix
\section{ }
%\section*{Appendix}

\begin{proof}[Proof of Lemma~\ref{lem:combinatorial-proof}]
Fix $k\in\N$, and consider the set $\mathcal{Z}=\{ z_{m,n} \mid p^{k-1}\le m<p^k,\,  n<m \}$. This set can be decomposed as
\[
\mathcal{Z}=\left(\bigcup_{1\le j\le i \le p-1} \mathcal{Z}_{i,j}\right) \quad \bigcup \quad \left(\bigcup_{1\le  i \le p-1} \mathcal{V}_i\right)
\]
where
\begin{align*}
    \mathcal{Z}_{i,j} &= \{ z_{m,n} \mid ip^{k-1}\le m<(i+1)p^{k-1},\, (j-1)p^{k-1}\le  n< jp^{k-1} \},\\
     \mathcal{V}_i &= \{ z_{m,n} \mid ip^{k-1}\le n<m<(i+1)p^{k-1} \}.
\end{align*}
It is helpful to visualise the elements of~$\mathcal{Z}$ as a grid where $m$ determines the row and $n$ the column, with $\mathcal{Z}_{i,j}$ representing the $p^{k-1}\times p^{k-1}$ squares if $j>1$ and the $p^{k-1}\times (p^{k-1}-1)$ squares if $j=1$,  and with $\mathcal{V}_i$  corresponding to the lower left triangles along the diagonal $m=n$.

Then, writing
\[
\mathcal{U}_{i,1}=\{ z_{m,n} %{\color{red}\mid ip^{k-1}\le m<(i+1)p^{k-1},\,n\le m-i p^{k-1}+1\},}
\in\mathcal{Z}_{i,1}\mid n\ge m-i p^{k-1}\},
\]
%\[
%\mathcal{V}_{i,1}=\{ z_{m,n} %{\color{red}\mid ip^{k-1}\le m<(i+1)p^{k-1},\,n\le m-i p^{k-1}+1\},}
%\in\mathcal{Z}_{i,1}\mid n\le m-i p^{k-1}+1\},
%\]
and
\[
\mathcal{W}_{i,j}=\{ z_{ip^{k-1},(j-1)p^{k-1}}\}% \mid ip^{k-1}\le m<(i+1)p^{k-1},\, n=(j-1)p^{k-1} \rangle
\quad\text{for }i\ge j>1,
\]
%\[
%\mathcal{W}_{i,j}=\langle z_{m,n} \mid ip^{k-1}\le m<(i+1)p^{k-1},\, n=(j-1)p^{k-1} \rangle \quad\text{for }i\ge j>1,
%\]}
it suffices to show that
\begin{align*}
&\big\langle [z_{m,n},x^{p^{k-1}},\overset{p-1}\ldots, x^{p^{k-1}}] \mid p^{k-1}\le m<p^k,\,  n<m \big\rangle\widetilde{\Theta}_k\\
&=\big\langle [z,x^{p^{k-1}},\overset{p-1}\ldots, x^{p^{k-1}}] \mid  z\in (\bigcup_{i=1}^{p-1}\mathcal{U}_{i,1} )\cup(\bigcup_{i=2}^{p-2}\mathcal{W}_{i,2})
\big\rangle\widetilde{\Theta}_k.
\end{align*}
% The second statement $[\widetilde{\Theta}_{k-1},x^{p^{k-1}},\overset{p-1}{\ldots},x^{p^{k-1}}]\widetilde{\Theta}_k=\widetilde{\Lambda}_k\widetilde{\Theta}_k$ is then immediate from this first statement of the lemma\comment{Maybe remove}.

Recall that, for $m\ge p^{k-1}$,
\[
[z_{m,n},x^{p^{k-1}},\overset{p-1}{\ldots},x^{p^{k-1}}]
\equiv \prod_{s=1}^{p-1}\left(\prod_{t=1}^{s}z_{m+(p-1-t)p^{k-1},n+(p-1-s+t)p^{k-1}}^{ \binom{p-1}{s}\binom{s}{t}}\right) \pmod {\widetilde{\Theta}_k};
\]
compare~\cite[Lem.~4.3]{HK}.
Note also that $\binom{p-1}{s}\equiv (-1)^s \pmod p$.

Now let $z\in \mathcal{Z}_{i,j}$ for fixed $i,j$. Writing $m=ip^{k-1}+\mu$ and $n=(j-1)p^{k-1}+\nu$ for $0\le \mu,\nu\le p^{k-1}-1$, we have
\begin{align*}
&[z_{ip^{k-1}+\mu,(j-1)p^{k-1}+\nu},x^{p^{k-1}},\overset{p-1}{\ldots},x^{p^{k-1}}]\\
&\qquad\qquad\qquad\equiv\prod_{s=i}^{p-1}\left(\prod_{t=i}^{s}z_{\mu+(p-1-t+i)p^{k-1},\nu+(p-1-(s-t)+(j-1))p^{k-1}}^{ (-1)^s\binom{s}{t}}\right)\pmod {\widetilde{\Theta}_k}.
\end{align*}
From the above it is clear that
for $\ell\le\frac{p-1}{2}$,
\[
\big\langle [z,x^{p^{k-1}},\overset{p-1}\ldots, x^{p^{k-1}}] \mid z\in\mathcal{Z}_{p-\ell,\ell+1}\cup\cdots\cup \mathcal{Z}_{p-\ell,p-\ell}\big\rangle\le \widetilde{\Theta}_k
\]
and similarly
\[
\big\langle [z,x^{p^{k-1}},\overset{p-1}\ldots, x^{p^{k-1}}] \mid z\in \mathcal{V}_{p-\ell}\big\rangle\le \widetilde{\Theta}_k.
\]

Now, set $\widetilde{\Theta}_{k,0}=\widetilde{\Theta}_k$
%\widetilde{\Theta}_{k,1}=\big\langle [z,x^{p^{k-1}},\overset{p-1}\ldots, x^{p^{k-1}}] \mid z\in\mathcal{V}_{p-1,1} \big\rangle \widetilde{\Theta}_k
%$$
and recursively define, 
for $1\le \tau\le p-2$, 
\[
\widetilde{\Theta}_{k,\tau}=\big\langle
[z,x^{p^{k-1}},\overset{p-1}\ldots, x^{p^{k-1}}] \mid z\in\mathcal{U}_{p-\tau,1} \cup \mathcal{W}_{p-\tau-1,2}
\big\rangle\widetilde{\Theta}_{k,\tau-1},
\]
where $\mathcal{W}_{1,2}=\varnothing$.

%{\color{red}Recall that for $z\in \mathcal{Z}_{i,j}$ we write $m=ip^{k-1}+\mu$ and $n=(j-1)p^{k-1}+\nu$ for $0\le \mu,\nu\le p^{k-1}-1$. For convenience, until further notice we will  assume that $\nu>0$.} 
We claim that for $1\le \tau\le p-1$,
\begin{equation}\label{eq:claim}
\begin{split}
&\big\langle [z,x^{p^{k-1}},\overset{p-1}\ldots, x^{p^{k-1}}] \mid z\in\mathcal{Z}_{p-\ell,\ell-\tau+1},\, \tau\le \ell\le \tfrac{p+\tau-1}{2} \big\rangle \widetilde{\Theta}_{k,\tau-1}\\
&\qquad\qquad=\big\langle [z,x^{p^{k-1}},\overset{p-1}\ldots, x^{p^{k-1}}] \mid z\in\mathcal{U}_{p-\tau,1} \cup \mathcal{W}_{p-\tau-1,2} \big\rangle \widetilde{\Theta}_{k,\tau-1}\\
&\qquad\qquad =\Big\langle\prod_{s=p-\tau}^{p-1}\prod_{t=p-\tau}^{s}z_{\mu+(p-1-t+p-\tau)p^{k-1},\nu+(p-1-s+t)p^{k-1}}^{ (-1)^s\binom{s}{t}}\mid \nu\ge\mu\Big\rangle\widetilde{\Theta}_{k,\tau-1}.
\end{split}
\end{equation}

Indeed, for $\tau=1$, the result is clear since
\begin{align*}
&[z_{(p-\ell)p^{k-1}+\mu,(\ell-1)p^{k-1}+\nu},x^{p^{k-1}},\overset{p-1}{\ldots},x^{p^{k-1}}]\\
&\qquad\qquad\qquad\equiv\prod_{t=p-\ell}^{p-1}z_{\mu+(p-1-t+p-\ell)p^{k-1},\nu+(\ell-1+t)p^{k-1}}^{ (-1)^{p-1}\binom{p-1}{t}}\pmod {\widetilde{\Theta}_k}\\
&\qquad\qquad\qquad\equiv z_{\mu+(p-1)p^{k-1},\nu+(p-1)p^{k-1}}^{ \binom{p-1}{p-\ell}}\pmod {\widetilde{\Theta}_k}.
\end{align*}

Suppose now that the result holds for $\tau-1$. Recall that for $z\in \mathcal{Z}_{i,j}$ we write $m=ip^{k-1}+\mu$ and $n=(j-1)p^{k-1}+\nu$ for $0\le \mu,\nu\le p^{k-1}-1$. We first assume that $\nu>0$, and we write $\mathcal{Z}_{i,j}^*$ for all such $z$ with $\nu>0$. %For convenience, until further notice we will  assume that $\nu>0$; a similar argument works for the case $\nu=0$.\footnote{\color{red}I hope ...}}
For   $z\in  \mathcal{Z}_{p-\ell,\ell-\tau+1}^*%{\color{red}\backslash \mathcal{W}_{p-\ell,\ell-\tau+1}}
$ with $\tau\le \ell\le \tfrac{p+\tau-1}{2}$, we obtain
\begin{align*}
&[z_{(p-\ell)p^{k-1}+\mu,(\ell-\tau)p^{k-1}+\nu},x^{p^{k-1}},\overset{p-1}{\ldots},x^{p^{k-1}}]\\
&\equiv\prod_{s=p-\tau}^{p-1}\left(\prod_{t=p-\ell}^{s}z_{\mu+(p-1-t+p-\ell)p^{k-1},\nu+(p+\ell-\tau-1-s+t)p^{k-1}}^{ (-1)^s\binom{s}{t}}\right)\pmod {\widetilde{\Theta}_{k}}\\
&\equiv\left(z_{\mu+(p-1)p^{k-1},\nu+(p-1)p^{k-1}}^{ (-1)^{p-\tau}\binom{p-\tau}{p-\ell}}\right)\left(z_{\mu+(p-1)p^{k-1},\nu+(p-2)p^{k-1}}^{ (-1)^{p-\tau+1}\binom{p-\tau+1}{p-\ell}} z_{\mu+(p-2)p^{k-1},\nu+(p-1)p^{k-1}}^{ (-1)^{p-\tau+1}\binom{p-\tau+1}{p-\ell+1}}\right)\\
&\qquad\times \cdots\times \left(\prod_{t=p-\ell}^{p-2-(\ell-\tau)}z_{\mu+(p-1-t+p-\ell)p^{k-1},\nu+(\ell-\tau+1+t)p^{k-1}}^{ -\binom{p-2}{t}}\right)\\
&\qquad\times \left(\prod_{t=p-\ell}^{p-1-(\ell-\tau)}z_{\mu+(p-1-t+p-\ell)p^{k-1},\nu+(\ell-\tau+t)p^{k-1}}^{ \binom{p-1}{t}}\right)\pmod {\widetilde{\Theta}_{k}}.
\end{align*}
We begin by establishing that
\begin{equation}\label{eq:claim2}
\begin{split}
&\big\langle [z,x^{p^{k-1}},\overset{p-1}\ldots, x^{p^{k-1}}] \mid z\in\mathcal{Z}_{p-\ell,\ell-\tau+1}^*%{\color{red}\backslash \mathcal{W}_{p-\ell,\ell-\tau+1}}
,\, \tau\le \ell\le \tfrac{p+\tau-1}{2} \big\rangle \widetilde{\Theta}_{k,\tau-1}\\
&\qquad\qquad=\big\langle [z,x^{p^{k-1}},\overset{p-1}\ldots, x^{p^{k-1}}] \mid z\in\mathcal{Z}_{p-\tau,1}  \big\rangle \widetilde{\Theta}_{k,\tau-1} %\\
%&\qquad\qquad =\Big\langle\prod_{s=p-\tau}^{p-1}\prod_{t=p-\tau}^{s}z_{\mu+(p-1-t+p-\tau)p^{k-1},\nu+(p-1-s+t)p^{k-1}}^{ (-1)^s\binom{s}{t}}\Big\rangle\widetilde{\Theta}_{k,\tau-1}.
\end{split}
\end{equation}
 If $\ell=
\tau$, the result is clear, so we assume that $\ell>\tau$.
Multiplying the above element with $[z_{(p-\ell+1)p^{k-1}+\mu,(\ell-\tau)p^{k-1}+\nu},x^{p^{k-1}},\overset{p-1}{\ldots},x^{p^{k-1}}]$ yields, modulo $\widetilde{\Theta}_{k}$,
\begin{align*}
    &\left(z_{\mu+(p-1)p^{k-1},\nu+(p-1)p^{k-1}}^{ (-1)^{p-\tau+1}\binom{p-\tau}{p-\ell+1}}\right)\left(z_{\mu+(p-1)p^{k-1},\nu+(p-2)p^{k-1}}^{(-1)^{p-\tau+2}\binom{p-\tau+1}{p-\ell+1}} z_{\mu+(p-2)p^{k-1},\nu+(p-1)p^{k-1}}^{(-1)^{p-\tau+2}\binom{p-\tau+1}{p-\ell+2}}\right)\\
&\qquad\times \cdots\times \left(\prod_{t=p-\ell+1}^{p-1-(\ell-\tau)}z_{\mu+(p-t+p-\ell)p^{k-1},\nu+(\ell-\tau+1+t)p^{k-1}}^{\binom{p-2}{t}}\right)\\
&\qquad\times \left(\prod_{t=p-\ell}^{p-1-(\ell-\tau)}z_{\mu+(p-1-t+p-\ell)p^{k-1},\nu+(\ell-\tau+t)p^{k-1}}^{ \binom{p-1}{t}}\right)\\
&\equiv\prod_{s=p-\tau}^{p-1}\left(\prod_{t=p-\ell+1}^{s}z_{\mu+(p-t+p-\ell)p^{k-1},\nu+(p-2+\ell-\tau-s+t)p^{k-1}}^{ (-1)^{s+1}\binom{s}{t}} \right) \pmod {\widetilde{\Theta}_{k}}\\
&\equiv[z_{(p-\ell+1)p^{k-1}+\mu,(\ell-\tau-1)p^{k-1}+\nu},x^{p^{k-1}},\overset{p-1}{\ldots},x^{p^{k-1}}]^{-1}  \pmod {\widetilde{\Theta}_{k}}.
\end{align*}
Therefore, as  $z_{(p-\ell+1)p^{k-1}+\mu,(\ell-\tau)p^{k-1}+\nu}\in\mathcal{Z}_{p-(l-1),(l-1)-(\tau-1)+1}$, repeating this process yields
\begin{equation}\label{eq:moving-up}
\begin{split}
&[z_{(p-\ell)p^{k-1}+\mu,(\ell-\tau)p^{k-1}+\nu},x^{p^{k-1}},\overset{p-1}{\ldots},x^{p^{k-1}}]\\
&\qquad\equiv[z_{(p-\ell+1)p^{k-1}+\mu,(\ell-\tau-1)p^{k-1}+\nu},x^{p^{k-1}},\overset{p-1}{\ldots},x^{p^{k-1}}]^{-1} \pmod {\widetilde{\Theta}_{k,\tau-1}}\\ &\qquad\,\,\, \vdots \\
&\qquad\equiv [z_{(p-\tau)p^{k-1}+\mu,\nu},x^{p^{k-1}},\overset{p-1}{\ldots},x^{p^{k-1}}]^{\pm 1} \pmod {\widetilde{\Theta}_{k,\tau-1}},
\end{split}
\end{equation}
In other words, %{\color{red}writing
%\[
%\mathcal{W}_{i,j}=\langle z_{m,n} \mid ip^{k-1}\le m<(i+1)p^{k-1},\, n=(j-1)p^{k-1} \rangle \quad\text{for }j>1,
%\]
%we obtain}
\begin{align*}
&\big\langle [z,x^{p^{k-1}},\overset{p-1}\ldots, x^{p^{k-1}}] \mid z\in\mathcal{Z}_{p-\tau,1} \big\rangle \widetilde{\Theta}_{k,\tau-1}\\
&\qquad\qquad=\big\langle [z,x^{p^{k-1}},\overset{p-1}\ldots, x^{p^{k-1}}] \mid z\in\mathcal{Z}_{p-\tau-1,2}^* \big\rangle \widetilde{\Theta}_{k,\tau-1}\\
&\qquad\qquad\,\,\,\vdots\\
&\qquad\qquad=\big\langle [z,x^{p^{k-1}},\overset{p-1}\ldots, x^{p^{k-1}}] \mid  z\in\mathcal{Z}_{\lceil\frac{p-\tau+1}{2}\rceil,\lfloor\frac{p-\tau+1}{2}\rfloor }^* \big\rangle \widetilde{\Theta}_{k,\tau-1},
\end{align*}
which yields~\eqref{eq:claim2}.

We consider now the case $\nu=0$, which only occurs for the squares $\mathcal{Z}_{i,j}$ where $j>1$.
Further it suffices to consider such elements %That is, for  
$z\in \mathcal{Z}_{p-\ell,\ell-\tau+1}$ with $2<\tau< \ell\le \tfrac{p+\tau-1}{2}$. Similarly one obtains
\begin{align*}
&\big\langle [z,x^{p^{k-1}},\overset{p-1}\ldots, x^{p^{k-1}}] \mid z\in\mathcal{Z}_{p-\tau-1,2} \backslash \mathcal{Z}_{p-\tau-1,2}^*\big\rangle \widetilde{\Theta}_{k,\tau-1}\\
&\qquad\qquad=\big\langle [z,x^{p^{k-1}},\overset{p-1}\ldots, x^{p^{k-1}}] \mid z\in \mathcal{Z}_{p-\tau-2,3} \backslash \mathcal{Z}_{p-\tau-2,3}^*\big\rangle \widetilde{\Theta}_{k,\tau-1}\\
&\qquad\qquad\,\,\,\vdots\\
&\qquad\qquad=\big\langle [z,x^{p^{k-1}},\overset{p-1}\ldots, x^{p^{k-1}}] \mid  \mathcal{Z}_{\lceil\frac{p-\tau+1}{2}\rceil-1,\lfloor\frac{p-\tau+1}{2}\rfloor+1}\backslash \mathcal{Z}_{\lceil\frac{p-\tau+1}{2}\rceil-1,\lfloor\frac{p-\tau+1}{2}\rfloor+1}^* \big\rangle \widetilde{\Theta}_{k,\tau-1},
\end{align*}
where of course $\mathcal{Z}_{i,j}=\varnothing$ if $j>i$.

%{\color{blue}
%Setting $\ell=\tau+1$ gives
%\begin{align*}
%&[z_{(p-\tau-1)p^{k-1}+\mu,p^{k-1}},x^{p^{k-1}},\overset{p-1}{\ldots},x^{p^{k-1}}]\\
%&\equiv\prod_{s=p-\tau}^{p-1}\left(\prod_{t=p-\tau-1}^{s}z_{\mu+(p-2-t+p-\tau)p^{k-1},(p-s+t)p^{k-1}}^{ (-1)^s\binom{s}{t}}\right)\pmod {\widetilde{\Theta}_{k}}\\
%&\equiv\left(z_{\mu+(p-1)p^{k-1},(p-1)p^{k-1}}^{ (-1)^{p-\tau}\binom{p-\tau}{p-\tau-1}}\right)\left(z_{\mu+(p-1)p^{k-1},(p-2)p^{k-1}}^{ (-1)^{p-\tau+1}\binom{p-\tau+1}{p-\tau-1}} z_{\mu+(p-2)p^{k-1},(p-1)p^{k-1}}^{ (-1)^{p-\tau+1}\binom{p-\tau+1}{p-\tau}}\right)\\
%&\qquad\times \cdots\times \left(\prod_{t=p-\tau-1}^{p-3}z_{\mu+(p-2-t+p-\tau)p^{k-1},tp^{k-1}}^{ -\binom{p-2}{t}}\right)\\
%&\qquad\times \left(\prod_{t=p-\tau-1}^{p-2}z_{\mu+(p-2-t+p-\tau)p^{k-1},(1+t)p^{k-1}}^{ \binom{p-1}{t}}\right)\pmod {\widetilde{\Theta}_{k}}.
%\end{align*}}

Hence we have the first step towards claim~\eqref{eq:claim}. It remains to show that
\begin{equation}
\label{eq: symmetry}
\begin{split}
\mathcal{U}_{\tau}&:=\big\langle [z,x^{p^{k-1}},\overset{p-1}\ldots, x^{p^{k-1}}] \mid z\in\mathcal{Z}_{p-\tau,1} \big\rangle \widetilde{\Theta}_{k,\tau-1}\\
&=\big\langle [z,x^{p^{k-1}},\overset{p-1}\ldots, x^{p^{k-1}}] \mid z\in\mathcal{U}_{p-\tau,1} \big\rangle \widetilde{\Theta}_{k,\tau-1}\\
&=\Big\langle\prod_{s=p-\tau}^{p-1}\prod_{t=p-\tau}^{s}z_{\mu+(p-1-t+p-\tau)p^{k-1},\nu+(p-1-s+t)p^{k-1}}^{ (-1)^s\binom{s}{t}} \mid \nu\ge \mu\Big\rangle\widetilde{\Theta}_{k,\tau-1}
\end{split}
\end{equation}
and 
\begin{equation}
\label{eq: W}
\begin{split}
&\big\langle [z,x^{p^{k-1}},\overset{p-1}\ldots, x^{p^{k-1}}] \mid z\in\mathcal{Z}_{p-\tau-1,2} \backslash \mathcal{Z}_{p-\tau-1,2}^* \big\rangle \mathcal{U}_{\tau}\\
&\qquad\qquad=\big\langle [z,x^{p^{k-1}},\overset{p-1}\ldots, x^{p^{k-1}}] \mid z\in\mathcal{W}_{p-\tau-1,2} \big\rangle \mathcal{U}_{\tau}.
\end{split}
\end{equation}

Suppose that $0<\nu<\mu$. 
Akin to \eqref{eq:moving-up}, we have that
\begin{align*}
    [z_{(p-\tau)p^{k-1}+\mu,\nu},x^{p^{k-1}},\overset{p-1}{\ldots},x^{p^{k-1}}]
    &\equiv  [z_{\lfloor\frac{p-\tau}{2}\rfloor p^{k-1}+\mu,\lceil\frac{p-\tau}{2}\rceil p^{k-1}+\nu},x^{p^{k-1}},\overset{p-1}{\ldots},x^{p^{k-1}}]^{\pm 1}\\
    &\equiv [z_{\lceil\frac{p-\tau}{2}\rceil p^{k-1}+\nu,\lfloor \frac{p-\tau}{2}\rfloor p^{k-1}+\mu},x^{p^{k-1}},\overset{p-1}{\ldots},x^{p^{k-1}}]^{\mp 1}\\
    &\equiv [z_{(p-\tau)p^{k-1}+\nu,\mu},x^{p^{k-1}},\overset{p-1}{\ldots},x^{p^{k-1}}]^{\epsilon},
    \end{align*}
where $\epsilon\in\{-1,1\}$, depending on whether $\tau$ is even or odd.
This proves (\ref{eq: symmetry}), and similarly for (\ref{eq: W}).

Finally, we  consider $\mathcal{V}_i$, for $1\le i\le \frac{p-1}{2}$. In a similar manner, one can show that
\[
\big\langle [z,x^{p^{k-1}},\overset{p-1}\ldots, x^{p^{k-1}}] \mid z\in\mathcal{V}_i \big\rangle \widetilde{\Theta}_{k,p-2i}= \widetilde{\Theta}_{k,p-2i},
\]
and the result follows.
%It also follows from the above that the generators given in the presentation of~$\widetilde{\Lambda}_k$ form a minimal generating set for~$\widetilde{\Lambda}_k$ modulo ~$\widetilde{\Theta}_k$. 
\end{proof}

We note that some of the generators in the presentation of $\widetilde{\Lambda}_k$ where $\mu=\nu$ are trivial modulo $\widetilde{\Theta}_k$.
The number of such generators is insignificant in relation to the minimum generating set of $\widetilde{\Lambda}_k$.

\end{document}